\documentclass[reqno]{amsart}

\makeatletter
\@namedef{subjclassname@2020}{\textup{2020} Mathematics Subject Classification}
\makeatother

\usepackage{amssymb}
\usepackage{graphicx}
\usepackage{mathrsfs}
\usepackage[colorlinks=true, linkcolor=blue, urlcolor=blue, citecolor=blue]{hyperref}
\usepackage{color}
\usepackage{url}
\usepackage{cite}

\usepackage{times}
\usepackage{microtype}
\usepackage{amssymb}
\usepackage{mathtools}
\usepackage{tikz}
\usepackage{wasysym}
\usepackage{centernot}
\usepackage{color}
\usepackage{appendix}

\newtheorem{theorem}{Theorem}[section]
\newtheorem{lemma}[theorem]{Lemma}
\newtheorem{proposition}[theorem]{Proposition}
\newtheorem{corollary}[theorem]{Corollary}

\theoremstyle{definition}
\newtheorem{definition}[theorem]{Definition}
\newtheorem{example}[theorem]{Example}

\theoremstyle{remark}
\newtheorem{remark}[theorem]{Remark}

\numberwithin{equation}{section}

\def\sC{\mathcal{C}}

\def\bP{\mathbb{P}}

\def\bR{{\mathbb R}}

\def\sE{{\mathscr E}}
\def\sF{{\mathscr F}}
\def\sA{{\mathscr A}}
\def\sG{{\mathscr G}}

\def\cF{\mathcal F}

\def\bH{\mathbf{H}}

\def\sL{\mathscr{L}}
\def\cD{\mathcal{D}}
\def\bH{\mathbf{H}}
\def\bD{\mathbf{D}}

\def\bE{\mathbb{E}}

\def\bN{\mathbb{N}}

\def\bE{\mathbb{E}}
\def\bQ{\mathbb{Q}}

\def\sB{\mathscr{B}}

\begin{document}

\title[On domination for (non-symmetric) Dirichlet forms]{On domination for (non-symmetric) Dirichlet forms}

\author{ Liping Li}
\address{Fudan University, Shanghai, China. }
\email{liliping@fudan.edu.cn}
\thanks{The first named author is a member of LMNS,  Fudan University.  He is partially supported by NSFC (No. 12371144). }

\author{Jiangang Ying}
\address{Fudan University, Shanghai, China. }
\email{jgying@fudan.edu.cn}



\date{\today}




\begin{abstract}
The primary aim of this article is to investigate the domination relationship between two $L^2$-semigroups using probabilistic methods. According to Ouhabaz's domination criterion, the domination of semigroups can be transformed into relationships involving the corresponding Dirichlet forms. Our principal result establishes the equivalence between the domination of Dirichlet forms and the killing transformation of the associated Markov processes, which generalizes and completes the results in \cite{Y962} and \cite{Y96}. Based on this equivalence, we provide a representation of the dominated Dirichlet form using the bivariate Revuz measure associated with the killing transformation and further characterize the sandwiched Dirichlet form within the broader Dirichlet form framework. In particular, our findings apply to the characterization of operators sandwiched between the Dirichlet Laplacian and the Neumann Laplacian. For the local boundary case, we eliminate all technical conditions identified in the literature \cite{AW03} and deliver a complete representation of all sandwiched operators governed by a Robin boundary condition determined by a specific quasi-admissible measure. Additionally, our results offer a comprehensive characterization of related operators in the non-local Robin boundary case, specifically resolving an open problem posed in the literature \cite{OA21}.
\end{abstract}

\maketitle

\tableofcontents

\section{Introduction}\label{SEC1}

Consider an open set $\Omega \subset \bR^n$. Let $\Delta^D$ and $\Delta^N$ denote the Laplace operators on $L^2(\Omega)$ subject to the Dirichlet boundary condition $u|_{\partial \Omega} = 0$ and the Neumann boundary condition $\frac{\partial u}{\partial n}|_{\partial \Omega} = 0$, respectively. Here, $L^2(\Omega)$ denotes the space of all square-integrable functions on $\Omega$,  $\partial \Omega$ represents the boundary of $\Omega$, and $\frac{\partial u}{\partial n}$ denotes the exterior normal derivative on the boundary (whose existence requires certain regularity assumptions). Denote by $(e^{t\Delta^D})_{t \geq 0}$ and $(e^{t\Delta^N})_{t \geq 0}$ the strongly continuous semigroups generated by $\Delta^D$ and $\Delta^N$, respectively.
In \cite{AW03}, Arendt and Warma investigated the question of which self-adjoint operator generates a semigroup $(T_t)_{t \geq 0}$ that is ``sandwiched" between $(e^{t\Delta^D})_{t \geq 0}$ and $(e^{t\Delta^N})_{t \geq 0}$ in the sense that
\begin{equation}\label{eq:11}
	e^{t\Delta^D}f\leq T_t  f\leq e^{t\Delta^N}f,\quad t\geq 0,f\in \mathrm{p}L^2(\Omega),
\end{equation}
where $\mathrm{p}L^2(\Omega)$ represents the family of all positive functions in $L^2(\Omega)$.
Their main result demonstrates that, under certain necessary conditions, $(T_t)_{t \geq 0}$ corresponds to the Laplace operator subject to a Robin boundary condition determined by a certain \emph{admissible} measure $\mu$ on $\partial \Omega$ (see Example~\ref{EXA75}). This operator and its semigroup are denoted by $\Delta_\mu$ and $e^{t\Delta_\mu}$, respectively. Specifically, when $\mu$ is absolutely continuous with respect to the surface measure on $\partial \Omega$ with density $\beta$, $\Delta_\mu$ corresponds to the classical Robin boundary condition:
\[
\beta u + \frac{\partial u}{\partial n} = 0\quad \text{on }\partial \Omega.
\]

Arendt and Warma's research builds upon Ouhabaz's work \cite{O96}, which provides a quadratic form characterization of the domination relation in \eqref{eq:11}. {(Indeed, similar characterizations have been previously discussed in the literature, including Silverstein \cite[\S21]{S74}.)} Specifically, if $(a^1, \cD(a^1))$ and $(a^2, \cD(a^2))$ are the closed forms associated with the strongly continuous semigroups $(T^1_t)_{t \geq 0}$ and $(T^2_t)_{t \geq 0}$, and both semigroups are positive (i.e., for any non-negative $f \in L^2(\Omega)$, $T^1_tf \geq 0$ and $T^2_tf \geq 0$), then
\begin{equation}\label{eq:13}
T^1_tf \leq T^2_tf,\quad \forall f\in \mathrm{p}L^2(\Omega)
\end{equation}
is equivalent to the conditions that $\cD(a^1)$ is an ideal of $\cD(a^2)$ and that
\[
a^1(u, v) \geq a^2(u, v), \quad u, v \in \mathrm{p}\cD(a^1),
\]
where $\mathrm{p}\cD(a^1)$ denotes the set of non-negative elements in $\cD(a^1)$ (see Lemma~\ref{LM21} for details).

In the context of \eqref{eq:11}, since the semigroups on both sides are positive and satisfy the (sub-)Markovian property (i.e., for any $f \in L^2(\Omega)$ with $0 \leq f \leq 1$, it holds that $0 \leq e^{t\Delta^D}f, e^{t\Delta^N}f \leq 1$), it follows that the sandwiched semigroup is also positive and Markovian. According to the foundational work of Beurling and Deny \cite{BD59}, semigroups with these properties correspond to closed forms known as \emph{Dirichlet forms}.

Over the past two decades, the study of sandwiched semigroups and quadratic forms has been advanced and applied from various perspectives. To illustrate, we highlight a selection of contributions. While much of the research continues to focus on Laplace operators (or elliptic differential operators), the scope of boundary conditions has expanded to include non-local settings, as explored in \cite{K18, OA21, VS12}. Additionally, the focus has broadened to encompass general Dirichlet forms, with notable works including \cite{ACD24, KL23, MVV05, R16, S20}. Recent years have also seen growing interest in extending these results to nonlinear semigroups and nonlinear Dirichlet forms, as exemplified by \cite{CC24}. Beyond these theoretical advances, related research has found applications in areas such as partial differential equations, fractal structures, and other fields, as demonstrated in \cite{BC17, GM24, SW10}.

It is important to note that, under conditions of \emph{regularity} or \emph{quasi-regularity}, the relationship between Dirichlet forms and Markov processes in probability theory is well established (see \cite{CF12, FOT11, MR92}). However, to our knowledge, aside from the mention of the probabilistic counterparts of Dirichlet forms in \cite{A12}, the aforementioned studies utilizing Ouhabaz's domination criterion have primarily adopted an analytical perspective, largely overlooking the substantial potential of probabilistic methods, despite most discussions occurring within the framework of regularity. Researchers familiar with Markov processes can readily relate the domination relationship of $L^2$-semigroups from \eqref{eq:13} to the \emph{subordination} between two Markov transition semigroups: $(Q_t)_{t\geq 0}$ is said to be \emph{subordinated to} $(P_t)_{t\geq 0}$ if
\begin{equation}\label{eq:12}Q_t f\leq P_t f,\quad \forall f\in \mathrm{p}\mathscr{B}.\end{equation}
This represents a classic topic in the theory of general Markov processes (see \cite[III]{BG68} and \cite[VII]{S88}). In simpler terms, the subordination stated in \eqref{eq:12} corresponds to the killing transformation of Markov processes. The Markov process corresponding to $(Q_t)_{t\geq 0}$, often referred to as the \emph{subprocess}, is always derived from the Markov process corresponding to $(P_t)_{t\geq 0}$ through a mechanism governed by some \emph{multiplicative functional} that terminates sample paths. The Dirichlet form characterization of general killing transformations has also been addressed in prior studies (such as \cite{Y96, Y962} {and \cite[\S21]{S74}}).  Nevertheless, it is crucial to acknowledge that there exists a significant disparity between domination in the sense of almost everywhere from \eqref{eq:13} and subordination in the pointwise sense from \eqref{eq:12}. Fortunately, the theory of regular Dirichlet forms established by Fukushima, along with the quasi-regular Dirichlet form framework developed by Ma et al. (see \cite{MR92}), is specifically designed to bridge this gap. Therefore, at least theoretically, there exists a probabilistic pathway to investigate domination in \eqref{eq:13}.

The primary objective of this paper is precisely to establish the equivalence between the domination \eqref{eq:13} of $L^2$-semigroups in analysis and the killing transformation of Markov processes in probability. Our principal result demonstrates the following intuitive fact: each dominated Dirichlet form is derived from two successive killing transformations. The first step involves terminating the Markov process when it exits a specific finely open set, which defines the actual state space of the subprocess corresponding to the dominated Dirichlet form. The second step consists of an additional killing transformation dictated by another multiplicative functional within the finely open set, which does not further alter the state space of the subprocess.  At the same time, these two steps of killing transformations correspond to the analytical decomposition of domination described in Theorem~\ref{THM24}: Every domination can be expressed as a composition of a Silverstein extension and strong subordination.

Extending beyond Ouhabaz’s domination criterion, our probabilistic characterization offers a representation of the dominated Dirichlet form using the bivariate Revuz measure of the multiplicative functional (which can be naturally extended to represent the sandwiched Dirichlet form, see Theorem~\ref{THM62}).  Specifically, for any (non-symmetric) Dirichlet forms $(a^1,\cD(a^1))$ and $(a^2,\cD(a^2))$ satisfying \eqref{eq:13}, there exists a positive measure $\sigma$ on $\overline{\Omega} \times \overline{\Omega}$, which can be expressed as detailed in \eqref{eq:46}, such that
\begin{equation}\label{eq:14}
\begin{aligned}
&\cD(a^1)=\cD(a^2)\cap L^2(\overline{\Omega},\bar{\sigma}),\\
& a^1(u,v)=a^2(u,v)+\sigma(u\otimes v),\quad u,v\in \cD(a^1),
\end{aligned}
\end{equation}
where $\bar{\sigma}$ denotes the marginal measure of $\sigma$. It is noteworthy that \cite{ACD24} has achieved a similar representation for regular and symmetric Dirichlet forms using analytical methods.

Our investigation is situated within the general framework of non-symmetric Dirichlet forms, with the essential condition being the quasi-regularity of the dominated Dirichlet form (see \S\ref{SEC5}). In other words, in the representation provided in \eqref{eq:14}, the quasi-regularity of $(a_1,\cD(a_1))$ is necessary, whereas $(a_2,\cD(a_2))$ does not need to satisfy the quasi-regularity condition. Consequently, the findings presented in this paper are applicable to the majority of situations discussed in the existing literature, including works such as \cite{A12, ACD24, KL23, MVV05, R16, S20}, among others. In particular, we re-examine the Laplacian-related sandwiched Dirichlet forms in \S\ref{SEC7}. Theorem~\ref{THM73} eliminates all assumptions imposed by Arendt and Warma \cite{AW03} concerning the characterization of sandwiched semigroups in \eqref{eq:11}: In fact, every semigroup $(T_t)_{t\geq 0}$ satisfying \eqref{eq:11} is determined by some \emph{quasi-admissible} measure $\mu$ associated with the Robin Laplacian. (The definition of quasi-admissibility can be found in Definition~7.1.) Moreover, comprehensive solutions can also be provided for Laplacian-related problems in the non-local Robin boundary case. Notably, our results address an open problem posed in the literature \cite{OA21}, as discussed in \S\ref{SEC73}.

Finally, we provide a brief description of commonly used symbols. The concepts and notations associated with Dirichlet forms align with those presented in \cite{CF12, FOT11, MR92}, among others.  For a function $u$ in a Dirichlet space,  $\tilde u$ denotes its quasi-continuous modification by default.  Given a topological space $E$,  $\mathscr{B}(E)$ represents the collection of Borel measurable sets or Borel measurable functions,  and $\mathscr B^*(E)$ denotes the collection of all universally measurable sets or universally measurable functions. For a class of functions $\mathscr{C}$,  $\mathrm{p} \mathscr{C}$ and $\mathrm{b} \mathscr{C}$ denote the families of non-negative elements and bounded elements, respectively, within $\mathscr{C}$.

\section{Decomposition of domination}

Let $E$ be a measurable space, and let $m$ be a $\sigma$-finite measure on $E$.  The norms of $L^2(E,m)$ and $L^\infty(E,m)$ are denoted by $\|\cdot\|_2$ and $\|\cdot\|_\infty$,  respectively, while the inner product of $L^2(E,m)$ is denoted by $(\cdot,\cdot)_m$.   A (non-symmetric) Dirichlet form $(\sE,\sF)$ on $L^2(E,m)$ is defined as a coercive closed  form that satisfies the Markovian property,  as described in,  e.g.,  \cite[I,  Definition~4.5]{MR92}.   The $\sE_1$-norm on $\sF$ is denoted by $\|\cdot\|_{\sE_1}$,  where $\|f\|_{\sE_1}:=\left(\sE(f,f)+(f,f)_m\right)^{1/2}$ for all $f\in \sF$.  The Dirichlet form is called symmetric if, in addition,  $\sE(f,g)=\sE(g,f)$ for all $f,g\in \sF$.  

For a given Dirichlet form $(\sE,\sF)$ on $L^2(E,m)$,  it is well known that $\mathrm{b}\sF:=\sF\cap L^\infty(E,m)$ is a subalgebra of $L^\infty(E,m)$ (see \cite[I, Corollary~4.15]{MR92}),  and $\sF$ forms a sublattice of $L^2(E,m)$ (see \cite[I, Proposition~4.11]{MR92}).  

If $U,V$ are subalgebras of $L^\infty(E,m)$,  we say that $U$ is an \emph{algebraic ideal} in $V$ if $f\in U$ and $g\in V$ imply $fg\in U$.  For sublattices $U, V$ of $L^2(E,m)$,  $U$ is called an \emph{order ideal} in $V$ if $f\in U,g\in V$, and $|g|\leq |f|$ imply $g\in U$.

For two Dirichlet forms $(\sE,\sF)$ and $(\sE',\sF')$,  we say that $\sE$ \emph{dominates} $\sE'$ (or equivalently,  $\sE'$ is \emph{dominated by} $\sE$) and write $(\sE',\sF')\preceq (\sE,\sF)$ or $\sE' \preceq \sE$ if either of the conditions in the following lemma is satisfied.
The second condition in this lemma is commonly referred to as the \emph{Ouhabaz domination criterion} (see \cite{O96}) for the domination of semigroups.


\begin{lemma}\label{LM21}
Let $(\sE,\sF)$ and $(\sE',\sF')$ be two Dirichlet forms on $L^2(E,m)$,  whose associated semigroups (as described in \cite[I, Remark~2.9(ii)]{MR92}) are denoted by $(T_t)_{t\geq 0}$ and $(T'_t)_{t\geq 0}$,  respectively.  The following assertions are equivalent:
\begin{itemize}
\item[(1)] $T'_t f\leq T_tf$ for all $t\geq 0$ and all non-negative $f\in L^2(E,m)$. 
\item[(2)] $\sF'\subset \sF$,  $\mathrm{b}\sF':=\sF'\cap L^\infty(E,m)$ is an algebraic ideal in $\mathrm{b}\sF$,  and 
\begin{equation}\label{eq:21}
	\sE'(f,g)\geq \sE(f,g)
\end{equation}
for all non-negative $f,g\in \sF'$.  
\item[(3)] $\sF'\subset \sF$,  $\sF'$ is an order ideal in $\sF$,  and \eqref{eq:21} holds for all non-negative $f,g\in \sF'$.
\end{itemize}
\end{lemma}
\begin{proof}
The equivalence between the first and third assertions was established in \cite[Corollary~4.3]{MVV05}.  The proof of (3)$\Rightarrow$(2) is straightforward.  It suffices to demonstrate that, under the assumptions of the third assertion,  if $f\in \mathrm{b}\sF'$ and $g\in \mathrm{b}\sF$,  then $fg\in \mathrm{b}\sF'$.  This follows immediately, as $|fg|\leq \|g\|_\infty |f|$ and $\sF'$ is an order ideal in $\sF$.  

The proof of (2)$\Rightarrow$(3) can be completed by repeating the argument in \cite[Lemma~2.2]{S20},  where only the symmetric case is considered.  For the convenience of the readers,  we restate the necessary details.  Let $f\in \sF$ and $g\in \sF'$ with $|f|\leq |g|$.  Our goal is to show that $f\in \sF'$.  {Since $\sF$ and $\sF'$ are lattices,  we first assume that $0\leq f\leq g$ and $f$ and $g$ are bounded.}  Let $A:=\{(x,y\in \mathbb{R}^2:0\leq x\leq y\}$ and $\varepsilon>0$.  Consider the function
\[
	C_\varepsilon:A\rightarrow \mathbb{R},\quad C_\varepsilon(x,y):=\frac{xy}{y+\varepsilon}.
\]
It satisfies the following inequality
\begin{equation}\label{eq:22}
	\left|C_\varepsilon(x_1,y_1)-C_\varepsilon(x_2,y_2)\right|\leq |x_1-x_2|+|y_1-y_2|
\end{equation}
for $(x_1,y_1),(x_2,y_2)\in A$.   The function $H_\varepsilon(y):=\frac{y}{y+\varepsilon}$ is Lipschitz continuous with $\|H'_\varepsilon\|_\infty<\frac{1}{\varepsilon}$ and $H_\varepsilon(0)=0$.   By \cite[I, Proposition~4.11]{MR92},  we have $H_\varepsilon(g)\in \mathrm{b}\sF'$.  Therefore,  by the second assertion,  $h_\varepsilon:=C_\varepsilon(f,g)=f\cdot H_\varepsilon(g)\in \mathrm{b}\sF'$.  It is evident that $h_\varepsilon\rightarrow f$ in $L^2(E,m)$ as $\varepsilon\downarrow 0$.  We now aim to show that
\[
	\sup_{\varepsilon>0}\sE'(h_\varepsilon,h_\varepsilon)<\infty,
\]
which, by the Banach-Saks theorem, will imply that $f\in \sF'$.  Indeed,  it follows from $0\leq h_\varepsilon\leq f\leq g$ and \eqref{eq:21} that
\[
\begin{aligned}
	\sE'(g,g)&=\sE'\left(h_\varepsilon+(g-h_\varepsilon),h_\varepsilon+(g-h_\varepsilon)\right) \\
	&\geq \sE'(h_\varepsilon,h_\varepsilon)+\sE(g,g)-\sE(h_\varepsilon,h_\varepsilon).
\end{aligned}\]
Using the property \eqref{eq:22} and applying \cite[I, Proposition~4.11]{MR92} to $(\sE,\sF)$,  we get
\[
	\sE(h_\varepsilon,h_\varepsilon)^{1/2}\leq \sE(f,f)^{1/2}+\sE(g,g)^{1/2}.
\]
Thus, we have
\[
\begin{aligned}
	\sup_{\varepsilon>0}\sE'(h_\varepsilon, h_\varepsilon)&\leq \sup_{\varepsilon>0}\sE(h_\varepsilon,h_\varepsilon)+\sE'(g,g)-\sE(g,g)  \\
	&\leq \left(\sE(f,f)^{1/2}+\sE(g,g)^{1/2} \right)^2+\sE'(g,g)-\sE(g,g)<\infty.
\end{aligned}\]
{Next we prove the assertion for the case that $0\leq f\leq g$.  By \cite[I, Proposition~4.17]{MR92}, we have $f_n:=f\wedge n\in \sF'$ for all $n\in \bN$.  Take $m\in \bN$ with $m>n$. Then we have $0\leq f_n\leq f_m\leq g_m:=g\wedge m$. By replacing $h_\varepsilon,f,g$ above with $f_n,f_m,g_m$, respectively, we can deduce that
\[
	\sE'(f_n,f_n)\leq \sE'(g_m,g_m)-\sE(g_m,g_m)+\sE(f_n,f_n)\leq \sE'(g,g)+\sE(f,f),
\]
which implies $f\in \sF'$. The proof for $f\in \sF'$ under the general condition $f\in \sF, g\in \sF'$ with $|f|\leq |g|$ can be obtained from this. 
}
This completes the proof.
\end{proof}

We introduce two additional notions.

\begin{definition}
Let $(\sE,\sF)$ and $(\sE',\sF')$ be two Dirichlet forms on $L^2(E,m)$. 
\begin{itemize}
\item[(1)] We say that $\sE'$ is \emph{subordinate to} $\sE$ if $\sF'\subset \sF$ and \eqref{eq:21} holds for all non-negative $f,g\in \sF'$.  Furthermore, $\sE'$ is \emph{strongly subordinate to} $\sE$ if, in addition,  $\sF'$ is dense in $\sF$ with respect to the $\sE_1$-norm.
\item[(2)] We say that $\sE$ is a \emph{Silverstein extension} of $\sE'$ if $\mathrm{b}\sF'$ is an algebraic ideal in $\mathrm{b}\sF$ and $\sE=\sE'$ on $\mathrm{b}\sF'\times \mathrm{b}\sF'$.  
\end{itemize}
\end{definition}
\begin{remark}
The notions of (strong) subordination were discussed in \cite{Y96}, where the author, under the additional assumption of quasi-regularity, examined the relationship between the (strong) subordination of Dirichlet forms and the killing transformation of the associated Markov processes. The second concept, introduced by Silverstein (see \cite{S74, S76}), is standard in the theory of (symmetric) Dirichlet forms. For more detailed information, readers are referred to \cite[\S6.6]{CF12}. However, it is important to emphasize that our discussion is not restricted to symmetric Dirichlet forms.
\end{remark}

Clearly,  $\sE'\preceq \sE$ if and only if $\sE'$ is subordinate to $\sE$ and $\mathrm{b}\sF'$ is an algebraic ideal in $\mathrm{b}\sF$ (or equivalently,  $\sF'$ is an order ideal in $\sF$).  By the proof of Lemma~\ref{LM21},  $\sE$ is a Silverstein extension of $\sE'$ if and only if $\sF'$ is an order ideal in $\sF$ and $\sE=\sE'$ on $\mathrm{b}\sF'\times \mathrm{b}\sF'$.  In particular,  if $\sE$ is a Silverstein extension of $\sE'$,  then 
$\sE'\preceq \sE$ and $\sE'$ is subordinate to $\sE$.  

We now present a simple yet intriguing result, which demonstrates that domination can be uniquely decomposed into a combination of strong subordination and Silverstein extension. {This result in the symmetric case has been established in \cite[Theorem~21.2]{S74}, and it is referred to as the Third Structure Theorem.}

\begin{theorem}\label{THM24}
Let $(\sE,\sF)$ and $(\sE',\sF')$ be two Dirichlet forms on $L^2(E,m)$ such that $\sE'\preceq \sE$.  Denote $\tilde{\sF}$ as the closure of $\sF'$ in $\sF$ with respect to the $\sE_1$-norm, and define
\[
	\tilde{\sE}(f,g):=\sE(f,g),\quad f,g\in \tilde{\sF}.
\]
Then, $\sE'$ is strongly subordinate to $\tilde{\sE}$, and $\sE$ is a Silverstein extension of $\tilde{\sE}$.  Furthermore, if $(\tilde\sE^1, \tilde\sF^1)$ is another Dirichlet form on $L^2(E, m)$ that satisfies these two conditions, then $(\tilde\sE^1,\tilde\sF^1)=(\tilde{\sE},\tilde{\sF})$.  
\end{theorem}
\begin{proof}
The first assertion that $\sE'$ is strongly subordinate to $\tilde{\sE}$ follows directly from the definition of $(\tilde{\sE},\tilde{\sF})$.  We will now demonstrate that $\sE$ is a Silverstein extension of $\tilde{\sE}$.  Consider $f\in \mathrm{b}\tilde{\sF}:=\tilde{\sF}\cap L^\infty(E,m)$ and $g\in \mathrm{b}\sF$.  Our objective is to establish that $fg\in \mathrm{b}\tilde{\sF}$.  Indeed,  we may take a sequence $f_n\in \mathrm{b}\sF'$ such that $\sE_1(f_n-f,f_n-f)\rightarrow 0$,  $\sup_{n\geq 1}\|f_n\|_\infty<\infty$, and $f_n$ converges to $f$,  $m$-a.e.  
Since $\mathrm{b}\sF'$ is an algebraic ideal in $\mathrm{b}\sF$,  it follows that $f_ng\in \mathrm{b}\sF'\subset \mathrm{b}\tilde{\sF}$.   Furthermore, 
according to \cite[Theorem~1.4.2]{FOT11},  we have
\[
	\sup_{n\geq 1}\tilde{\sE}_1(f_ng,f_ng)^{1/2}\leq \sup_{n\geq 1}\|f_n\|_\infty\cdot  \sqrt{\sE_1(g,g)}+\|g\|_\infty \sup_{n\geq 1}\sqrt{\sE_1(f_n,f_n)}<\infty.
\]
By the Banach-Saks theorem,  there exists a subsequence $\{f_{n_k}g\}$ of $\{f_ng\}$ and $h\in \tilde{\sF}$ such that 
\[
	\frac{1}{N}\sum_{k=1}^N f_{n_k}g\rightarrow h,\quad N\rightarrow \infty
\]
with respect to the $\tilde{\sE}_1$-norm.  Note that $f_n\rightarrow f$,  $m$-a.e.  Therefore,  $h=fg$,  $m$-a.e.,  confirming that $fg\in \mathrm{b}\tilde{\sF}$.

Let $(\tilde{\sE}^1,\tilde{\sF}^1)$ be another Dirichlet form that satisfies the two specified conditions.  Since $\sE$ is a Silverstein extension of $\tilde{\sE}^1$,  it follows that $\tilde{\sE}^1(f,g)=\sE(f,g)$ for all $f,g\in \tilde{\sF}^1$.  The strong subordination of $\sE$ to $\tilde{\sE}^1$ implies that $\sF'$ is dense in $\tilde{\sF}^1$ with respect to the $\sE_1$-norm.  Consequently,   we have $\tilde{\sF}=\tilde{\sF}^1$ and $\tilde{\sE}(f,g)=\sE(f,g)=\tilde{\sE}^1(f,g)$ for all $f,g\in \tilde{\sF}=\tilde{\sF}^1$.  This completes the proof.
\end{proof}

\section{Killing is domination}

The relationship between Dirichlet forms and Markov processes is commonly known as contingent upon the condition of \emph{quasi-regularity}. In Appendix~\ref{SEC3}, we present an overview of the content relevant to this condition and the properties of the corresponding Markov processes.


Let $E$ be a Lusin topological space {with $\sB(E)=\sigma(C(E))$} and $m$ a $\sigma$-finite measure on $E$ with support $\text{supp}[m]=E$. 
We consider a quasi-regular Dirichlet form $(\sE,\sF)$ on $L^2(E,m)$.  Let $X=\left(\Omega, \mathcal{F}, \mathcal{F}_t, X_t, \theta_t,\bP_x\right)$ with lifetime $\zeta$ be an $m$-tight Borel right process properly associated with $(\sE,\sF)$.  The cemetery is denoted by $\Delta$.  We may assume that $\Omega$ contains a distinguished point $[\Delta]$ such that $X_t([\Delta])=\Delta$ for all $t\geq 0$.  The goal of this section is to demonstrate that the killing transformation on $X$ induces the domination of $(\sE,\sF)$.

\subsection{Part process}\label{SEC42}


We first investigate a specific killing transformation that results in a part process defined on a finely open set.

Let $G\subset E$ be a finely open set,  which is nearly Borel measurable since $X$ is {a Borel right process}, and define $B:=E\setminus G$.  We denote $D_B:=\inf\{t\geq 0:X_t\in B\}$,  and for $\omega\in \Omega$,  we define
\[
	X^G_t(\omega):=\left\lbrace
	\begin{aligned}
		&X_t(\omega),\quad &0\leq t<D_B(\omega)\wedge \zeta(\omega), \\
		&\Delta,\quad &t\geq D_B(\omega) \wedge \zeta(\omega),
	\end{aligned}	
	\right. \qquad \zeta^G(\omega):=D_B(\omega)\wedge \zeta(\omega),
\]
and
\[
	\theta^G_t(\omega):=\left\lbrace\begin{aligned}
	&\theta_t(\omega),\quad &t<\zeta^G(\omega),\\
	&[\Delta],\quad &t\geq \zeta^G(\omega).
	\end{aligned}\right.
\]
According to \cite[(12.24)]{S88},  
\begin{equation}\label{eq:30}
X^G:=(\Omega, \cF,\cF_t, X^G_t, \theta^G_t, \bP_x)
\end{equation} is a right process on the Radon space $(G,\sB^*(G))$ with lifetime $\zeta^G$.  This process is commonly referred to as the \emph{part process} of $X$ on $G$,  and its Dirichlet form,  as expressed in \eqref{eq:31}, is called the \emph{part Dirichlet form} of $(\sE,\sF)$ on $G$.  


\begin{lemma}\label{LM42}
The right process $X^G$ is $m|_{G}$-tight and is properly associated with the quasi-regular Dirichlet form given by
\begin{equation}\label{eq:31}
	\begin{aligned}
		&\sF^G:=\{f\in \sF: \tilde{f}=0,\sE\text{-q.e. on }B\},\\
		&\sE^G(f,g):=\sE(f,g),\quad f,g\in \sF^G
	\end{aligned}
\end{equation}
on $L^2(G,m|_G)$,  where $\tilde{f}$ denotes the $\sE$-quasi-continuous $m$-version of $f$. 
\end{lemma}
\begin{proof}
This characterization has already been established in \cite[Theorem~5.10]{F01}.  However, in the discussion of that theorem, tightness is established in the space \eqref{eq:35},  which is slightly larger than $G$. {Quasi-regularity is a direct consequence of tightness, as stated in \cite[Theorem~3.22]{F01}; see also \cite[Proposition~3.2~(i)]{K08}.} In the following, we will adopt a similar approach to that discussed in \cite[Theorem~3.3.8]{CF12} to prove the tightness of $X^G$ within $G$.

Define $u(x):=1-\bP_x e^{-T_B}$ for $x\in E$.  Then $u$ is finely continuous, and hence $\sE$-quasi-continuous by Lemma~\ref{LM32}~(3).  It follows that $\{u>0\}=E\setminus B^r=G\cup (B\setminus B^r)$,  where $B^r$ denotes the set of all regular points for $B$.  Note that $N:=B\setminus B^r$ is semipolar,  and hence also $\sE$-polar.  Consider an $\sE$-nest $\{K_n\}$ consisting of compact sets such that $\bigcup_{n\geq 1}K_n\subset E\setminus N$ and $u\in C(\left\{K_n\right\})$.  Let $F_n:=\{u\geq 1/n\}$ and $K^G_n:=K_n\cap F_n$.  It is straightforward to verify that $K^G_n$ is a compact subset of $G$.  Define $T^G_n:=\inf\{t>0:X^G_t\in G\setminus K^G_n\}$.  Note that $T^G_n\geq T_{G\setminus K_n}\wedge T_{G\setminus F_n}$.  For $m$-a.e. $x\in G$,  we have
\[
\begin{aligned}
	\left\{\lim_{n\rightarrow \infty}T^G_n<\zeta^G\right\}&\subset \left\{\lim_{n\rightarrow \infty}T_{G\setminus K_n}<T_B\wedge \zeta\right\}\cup \left\{\lim_{n\rightarrow \infty}T_{G\setminus F_n}<T_B\wedge \zeta\right\} \\
	&\subset  \left\{\lim_{n\rightarrow \infty}T_{E\setminus K_n}< \zeta\right\}\cup \left\{\lim_{n\rightarrow \infty}T_{G\setminus F_n}<T_B\wedge \zeta\right\}.  
\end{aligned}\]
By the $m$-tightness of $X$,  it suffices to demonstrate that for $m$-a.e.  $x\in G$,
\[
	\bP_x\left(\lim_{n\rightarrow \infty}T_{G\setminus F_n}<T_B\wedge \zeta\right)=0.
\]
In fact,  it follows from the quasi-left-continuity of $X$ and \cite[(4.8)]{F01} that
\[
	\bP_x\left(\lim_{n\rightarrow \infty}T_{G\setminus F_n}<T_B\wedge \zeta\right)=\bP_x\left(u(X_T)=\lim_{n\rightarrow \infty}u(X_{T_{G\setminus F_n}}), T<T_B\wedge \zeta \right),
\]
where $T:=\lim_{n\rightarrow \infty}T_{G\setminus F_n}$.  Since $X_{T_{G\setminus F_n}}\in \left(G\setminus F_n\right)^r\subset \{u\leq 1/n\}$,  we obtain
\[
\bP_x\left(\lim_{n\rightarrow \infty}T_{G\setminus F_n}<T_B\wedge \zeta\right)=\bP_x\left(u(X_T)=0,T<T_B\wedge \zeta\right)=0.
\]
This completes the proof.
\end{proof}

It is important to examine the relationship between $\sE^G$-quasi-notion and $\sE$-quasi-notion. 
Let $\{F_n:n\geq 1\}$ be an $\sE$-nest, and define $F^G_n:=F_n\cap G$ for $n\geq 1$.  Since
\[
	T^G_{G\setminus F^G_n}:=\inf\{t>0: X_t\in (E\setminus F_n)\cap G\}\geq T_{E\setminus F_n},
\]
it follows from Lemma~\ref{LM32}~(2) that $\{F^G_n:n\geq 1\}$ is an $\sE^G$-nest.  In particular,  the restriction of an $\sE$-quasi-continuous function to $G$ is $\sE^G$-quasi-continuous,  and $N\cap G$ is $\sE^G$-polar if $N$ is $\sE$-polar.

{We present another two facts that will be utilized subsequently.

\begin{corollary}\label{COR33}
\begin{itemize}
\item[(1)] A set $N\subset G$ is $\sE^G$-polar if and only if it is $\sE$-polar.
\item[(2)] Let $G_1\subset G$.  Then, $G_1$ is $\sE$-quasi-open  if and only if it is $\sE^G$-quasi-open.
\end{itemize}
\end{corollary}
\begin{proof}
See \cite[Proposition~3.2~(ii)]{K08}. 
\end{proof}}

\subsection{Multiplicative functionals and killing transformation}\label{SEC31}

As  a significant type of transformation for Markov processes (see \cite[III]{BG68}), 
the general killing transformation is defined through the use of decreasing multiplicative functionals.  It is important to note that the formulation in this subsection is applicable not only to Borel right processes but also to general right processes on a Radon topological space; see, e.g., \cite[Chapter VII]{S88}.

\begin{definition}\label{DEF33}
A real-valued process $M=(M_t)_{t\geq 0}$ is called a \emph{multiplicative functional} (MF, for short) of $X$ if it satisfies the following conditions:
\begin{itemize}
\item[(1)] The map $t\mapsto M_t(\omega)$ is decreasing,  right continuous and takes values in $[0,1]$ for all $\omega\in \Omega$;
\item[(2)] $M$ is adapted, i.e., $M_t\in \mathcal{F}_t$ for every $t\geq 0$;
\item[(3)] $M_{t+s}(\omega)=M_t(\omega)M_s(\theta_t\omega)$ holds for all $s,t\geq 0$ and $\omega\in \Omega$.
\end{itemize}
\end{definition}
\begin{remark}
In the conventional definition of an MF,  as presented in \cite[III, Definition~1.1]{BG68}, the third condition is expressed in an \emph{a.s.} sense.  That is, for  $t, s \geq 0$  and any  $x \in E$, the set  $\{\omega \in \Omega : M_{t+s}(\omega) \neq M_t(\omega)M_s(\theta_t \omega)\}$  is a $\bP_x$-null set. This type of definition is referred to as a \emph{weak} MF in \cite[\S54]{S88}.  In contrast, an MF that satisfies  $\bigcup_{s,t \geq 0} \{\omega \in \Omega : M_{t+s}(\omega) \neq M_t(\omega)M_s(\theta_t \omega)\} = \emptyset$ is called a \emph{perfect} MF. In fact, since a weak MF $M$ always admits a \emph{perfect regularization} $\bar{M}$ (see \cite[(55.19)]{S88}), it follows that $M$ is equivalent to a perfect MF $\bar{M}\cdot 1_{[0,D_{E\setminus E_M})}$ in the sense that
\[
	M_t=\bar{M}_t\cdot 1_{[0,D_{E\setminus  E_M})}(t),\quad \forall t\geq 0, \; \text{a.s.}
\]
provided that the set $E_M$, defined below, is nearly Borel measurable. 
\end{remark}

For an MF $M$,  we define $E_M:=\{x\in E: \bP_x(M_0=1)=1\}$,  the set of permanent points of $M$,  and $S_M:=\inf\{t>0:M_t=0\}$, the lifetime of $M$.  There is no loss of generality in assuming that $M_t=0$ for $t\geq \zeta$ since we can always replace $M$ with another \emph{equivalent} right continuous MF,  given by $M_t1_{[0,\zeta)}(t)$.  Particularly,  $M_t([\Delta])=0$ for all $t\geq 0$.  In addition, note that if $E_M$ is nearly Borel measurable and $D_{E\setminus E_M}:=\inf\{t\geq 0: X_t\in E\setminus E_M\}<\infty$,  then $M_{D_{E\setminus E_M}}=0$ and, in particular,  $S_M\leq D_{E\setminus E_M}$; see \cite[(57.2)]{S88}.  

 Let $\text{MF}(X)$ denote the family of all MFs $M$ of $X$ such that $E_M$ is finely open and $M_\zeta=0$. Further,  we define
\[
	\text{MF}_+(X):=\{M\in \text{MF}(X): E_M=E\},\quad \text{MF}_{++}(X):=\{M\in \text{MF}(X)_+: S_M=\zeta\}. 
\]
Since {a finely open set is \emph{optional} (see \cite[Theorem~8.6~(iii)]{S88})},  every $M\in \text{MF}(X)$ is a \emph{right} MF in the sense of \cite[(57.1)]{S88}.  In this context,  we can define a probability measure $\mathbb{Q}_x$ on $(\Omega, \mathcal F^0)$ for each $x\in E_M$  by
\begin{equation}\label{eq:33}
	\bQ_x(Z):=\bP_x \int_{(0,\infty]}Z\circ k_t d(-M_t),\quad Z\in \text{b}\cF^0,
\end{equation}
where $M_\infty:=0$,  $(k_t)_{t\geq 0}$ are the killing operators on $\Omega$ defined by $k_t\omega(s):=\omega(s)$ if $t>s$ and $k_t\omega(s)=\Delta$ if $t\leq s$,  and $\cF^0$ is the $\sigma$-algebra generated by the maps $f(X_t)$ with $t\geq 0$ and $f\in \sB(E)$.  According to,  e.g.,  \cite[(61.5)]{S88},  the restriction of $(\Omega, X_t,  \bQ_x)$ to $E_M$ is a right process on $E_M$,  referred to as the ($M$-)subprocess of $X$ and denoted by $X^M$ or $(X,M)$.   A more intuitive construction for subprocesses can also be found in \cite[III\S3]{BG68}. Clearly, the transition semigroup $(Q_t)_{t\geq 0}$ and the resolvent $(V_\alpha)_{\alpha>0}$ of $X^M$ are given by
\[
\begin{aligned}
	&Q_t f(x)=\bP_x \left(f(X_t)M_t\right), \\
	&V_\alpha f(x)=\bP_x \int_0^\infty e^{-\alpha t}f(X_t)M_tdt
\end{aligned}\]
for $f\in \sB(E_M)$ and $x\in E_M$.  Note that $Q_tf\in \sB^*(E_M)$, but it is not necessarily Borel measurable even if $E_M$ is Borel measurable. 

\begin{remark}
An MF $M$ is termed \emph{exact} if for any $t>0$ and every sequence $t_n\downarrow 0$, it holds that $M_{t-t_n}\circ \theta_{t_n}\rightarrow M_t$,  a.s., as $n\rightarrow\infty$.  If $M$ is exact,  then $E_M$ is finely open,  implying that $M\in \text{MF}(X)$; see \cite[(56.10)]{S88}.  Furthermore, every $M\in \text{MF}_+(X)$ is exact; see \cite[III, Corollary~4.10]{BG68}. 

It should be noted that in \cite{Y962},  $\text{MF}(X)$ represents all exact MFs of $X$, which differs from its usage in this paper.  In fact, for any finely closed set $B$, $1_{[0,D_B \wedge \zeta)}$ is an MF of $X$, with the set of its permanent points being the {finely open} set $G=E\setminus B$.  Namely,  $1_{[0,D_B\wedge \zeta)}\in \text{MF}(X)$ in our context. However,  $1_{[0,D_B\wedge \zeta)}$ is not exact,  as $\lim_{t\downarrow 0}(t+D_B\circ \theta_t)=T_B$ (not $D_B$; see \cite[(10.4)]{S88}). 

It is also important to note that another MF $1_{[0,T_B\wedge \zeta)}$ defined by the first hitting time $T_B$ is typically exact; however, its set of permanent points is
\begin{equation}\label{eq:35}
E_G:=\left\{x\in E: \bP_x(T_B=0)=0\right\}=G \cup (B\setminus B^r),
\end{equation}
which is not necessarily equal to $G$.  Killing $X$ by $1_{[0,T_B\wedge \zeta)}$ can also be understood as killing the part process $X^{E_G}$ using $1_{[0,T_B\wedge \zeta)}\in \text{MF}_+(X^{E_G})$.  This killing transformation yields a right process,  denoted by $X^{T_B}$, on $E_G$, as examined in \cite[Theorem~5.10]{F01}.  It is noteworthy that $E_G\setminus G$ is $\sE$-polar,  and hence $m|_{E_G}$-inessential for $X^{T_B}$.  In particular,  the restriction of $X^{T_B}$ to $G$ is precisely the part process $X^G$.
\end{remark}

We present a simple observation for future reference.  Recall that $X^G$,  as expressed in \eqref{eq:30}, denotes the part process of $X$ on $G$ when $G$ is finely open.

\begin{lemma}\label{LM36}
If $M\in \text{MF}(X)$, then $M \in \text{MF}_+(X^{E_M})$, and the subprocess $(X,M)$ is the same as the subprocess $(X^{E_M},M)$.   Conversely,  if $G$ is finely open and $M\in \text{MF}_+(X^G)$,  then $M\cdot 1_{[0,D_{E\setminus G})}\in \text{MF}(X)$,  whose set of permanent points is $G$,  and the subprocess $(X^G,M)$ is the same as $(X, M\cdot 1_{[0,D_{E\setminus G})})$.  
\end{lemma}
\begin{proof}
It is straightforward to verify the conditions of Definition~\ref{DEF33} to obtain that if $M\in \text{MF}(X)$,  then $M\cdot 1_{[0,D_{E\setminus E_M})}$ is an MF of $X^{E_M}$. The set of its permanent points is clearly $E_M$,  since $D_{E\setminus E_M}>0$,  $\bP_x$-a.s. for all $x\in E_M$.  Note that $M_\zeta=0$ and $M_{D_{E\setminus E_M}}=0$.  Thus,  $M_{\zeta^{E_M}}=M_{\zeta\wedge D_{E\setminus E_M}}=0$.  This implies that $M=M\cdot 1_{[0,D_{E\setminus E_M})} \in \text{MF}_+(X^{E_M})$.  
To establish the identification between $(X,M)$ and $(X^{E_M},M)$,  we observe that 
\[
	\hat{k}_t:=k_{t\wedge D_{E\setminus E_M}},\quad t\geq 0
\] 
are the killing operators of $X^{E_M}$.  The probability measure $\hat{\bQ}_x$ of $(X^{E_M},M)$ on $(\Omega,\cF^0)$ for $x\in E_M$ is given by
\[
	\hat{\bQ}_x(Z)=\bP_x\int_{(0,\infty]}Z\circ \hat{k}_t d(-M_t),\quad Z\in \mathrm{b}\cF^0. 
\]
It follows from $M_{D_{E\setminus E_M}}=0$ that
\[
\begin{aligned}
	\hat{\bQ}_x(Z)&=\bP_x\int_{(0,D_{E\setminus E_M}]} Z\circ \hat{k}_t d(-M_t)=\bP_x\int_{(0,D_{E\setminus E_M}]} Z\circ k_t d(-M_t)\\
	&=\bP_x\int_{(0,\infty]}Z\circ k_td(-M_t)=\bQ_x(Z).
\end{aligned}
\]
According to \cite[(61.5)]{S88},  we can obtain the desired conclusion.

The converse part can be demonstrated in a similar manner.
\end{proof}
\begin{remark}\label{RM37}
An $(\cF_t)$-stopping time $T$ is called a \emph{(perfect) terminal time} if the set $\{\omega\in \Omega: t+T(\theta_t\omega)\neq T(\omega),\exists t<T(\omega)\}$ is empty; see \cite[(12.1)]{S88}.  Typical examples of terminal times include the first hitting time $T_B$ and the first entrance time $D_B$ for any nearly Borel set $B$.  Consider a terminal time $T$ such that  $0<T\leq \zeta$.  We have $1_{[0,T)}\in \text{MF}_+(X)$,  and we denote  by $X^T$  (with lifetime $\zeta^T=T$) the subprocess of $X$ that is killed by $1_{[0,T)}$; see \cite[(12.23)]{S88}.  By employing a similar argument, it is straightforward to verify the following facts:
\begin{itemize}
\item[(1)] If $M$ is an MF of $X$,  then $M\cdot 1_{[0,T)}$ is an MF of $X^T$ with the same set of permanent points.  
\item[(2)] For $M\in \text{MF}_+(X)$ and $T:=S_M$,  it holds that $M\in\text{MF}_{++}(X^T)$. Furthermore,  the subprocess $(X,M)$ is the same as $(X^T,M)$. 
\end{itemize}
Additionally,  according to \cite[Theorem~2.2]{Y962},  every $M\in \text{MF}_+(X)$ admits the decomposition
\begin{equation}\label{eq:410}
	M_t=\prod_{0<s\leq t}\left(1-\Phi(X_{s-},X_{s})\right) \exp\left\{-\int_0^t a(X_s)dA_s\right\}1_{[0,J_\Gamma\wedge \zeta)}(t),
\end{equation}
where $\Phi\in \sB(E\times E), 0\leq \Phi<1$,  $\Phi$ vanishes on the diagonal $d$ of $E\times E$,  $a\in \mathrm{p}\sB(E)$,  $A$ is a positive continuous additive functional of $X$,  $\Gamma\in \sB(E\times E)$ is disjoint from $d$ such that $S_M=J_\Gamma\wedge \zeta$, where $J_\Gamma:=\inf\{t>0:(X_{t-},X_t)\in \Gamma\}>0$,  a.s.  (Every exact terminal time admits a representation of the form $J_\Gamma$; see \cite[Theorem~6.1]{S71}.)  Note that the product and the integral in \eqref{eq:410} can diverge only at $t\geq J_{\Gamma}$.  See also \cite[Theorem~7.1]{S71}.  
\end{remark}


\subsection{Bivariate Revuz measures of MFs}\label{SEC43}

For an MF $M$ of $X$,  we define its \emph{bivariate Revuz measure} $\nu_M$ (relative to $m$) by
\begin{equation}\label{eq:44}
	\nu_M(F):=\lim_{t\downarrow 0}\frac{1}{t}\bP_m \int_0^t F(X_{s-},X_s)d(-M_s),\quad F\in \mathrm{p}\sB(E\times E).
\end{equation}
The existence of $\nu_M$ is referred to in \cite{FG88}.  (In fact, $t\mapsto \int_0^t F(X_{s-},X_s)d(-M_s)$ is a raw additive functional of $X^M$ in the sense of \cite[Definition~2.3]{FG88}.  Hence, the limit in \eqref{eq:44} exists in $[0,\infty]$,  as demonstrated in \cite[Proposition~2.5]{FG88}).   For $f,g\in \sB(E)$,  we write $f(x)g(y)$ as $f\otimes g$.  Denote by $\rho_M$ and $\lambda_M$ the marginal measures of $\nu_M$,  i.e.,  for any $f\in \mathrm{p}\sB(E)$,
\[
	\rho_M(f):=\nu_M(1\otimes f),\quad \lambda_M(f):=\nu_M(f\otimes 1). 
\]
Note that any function defined on  $E$  is extended to  $E \cup \{\Delta\}$  by setting  $f(\Delta)=0$. Similarly, for $F \in \sB(E \times E)$, whenever  $x$  or  $y$  is  $\Delta$, we define  $F(x, y) := 0$.  Particularly,  the actual integration interval in \eqref{eq:44} is $[0, t] \cap [0, S_M] \cap [0, \zeta)$.

The significance of bivariate Revuze measures lies in the representation of the Dirichlet form associated with the subprocess $X^M$.  Let $m^*:=m|_{E_M}$.  The process $X^M$ is called \emph{nearly $m^*$-symmetric} if its transition semigroup $(Q_t)_{t\geq 0}$ acts on $L^2(E_M,m^*)$ as a strongly continuous contraction semigroup,  and its infinitesimal generator $\sL$ with domain $\cD(\sL)\subset L^2(E_M,m^*)$ satisfies the \emph{sector condition}.  This condition states that there exists a finite constant $K\geq 1$ such that 
\[
	\left|(-\sL f,g)_{m^*}\right|\leq K\cdot (-\sL f,f)_{m^*}^{1/2}\cdot (-\sL g,g)_{m^*}^{1/2}
\]
for all $f,g\in \cD(\sL)$.  The following theorem characterizes the Dirichlet form of $X^M$ for the case where $M \in \text{MF}_{++}(X)$. Its general form will be presented in Theorem~\ref{THM33}.

\begin{theorem}\label{THM43}
Let $M\in \text{MF}_{++}(X)$.  Then the subprocess $X^M$ is nearly $m$-symmetric,  and the associated Dirichlet form on $L^2(E,m)$ is given by
\[
\begin{aligned}
	&\sF^M=\sF\cap L^2(E,\rho_M+\lambda_M),\\
	&{\sE^M(f,g)}=\sE(f,g)+\nu_M(\tilde f\otimes \tilde g),\quad f,g\in \sF^M.
\end{aligned}\]
\end{theorem}
\begin{proof}
See \cite[Theorem~3.10]{Y962}.
\end{proof}

We present a characterization of the bivariate Revuz measure $\nu_M$ for $M\in \text{MF}_+(X)$, expressed as in \eqref{eq:410}.    Let $(N,H)$ be the L\'evy system of $X$ (see \cite[\S73]{S88}).  Note that $H$ is a positive continuous additive functional of $X$,  whose Revuz measure with respect to $m$ is denoted by $\rho_H$.  As established in \cite[Theorem~4.6]{Y962}, the following representation of $\nu_M$ holds:
\begin{equation}\label{eq:46}
	\nu_M(dxdy)=\left(1_\Gamma+1_{\Gamma^c}\cdot \Phi\right)(x,y)\nu(dxdy)+\delta_{y}(dx)a(y)\rho_A(dy),
\end{equation}
where {$\Gamma, \Phi, a, A$ are given in \eqref{eq:410},  $\rho_A$ stands for the Revuz measure of $A$,} $\nu(dxdy):=N(x,dy)\rho_H(dy)$ is referred to as the \emph{canonical measure} of $X$ (off the diagonal $d$),  and $\delta_y$ denotes the Dirac measure at $y$.  

\begin{remark}
If $(\sE,\sF)$ is symmetric,  then the canonical measure $\nu$ of $X$ coincides exactly with the jumping measure of $(\sE,\sF)$,  as stated in \cite[Theorem~4.3.3 and Proposition~6.4.1]{CF12}.  When $(\sE,\sF)$ is regular but non-symmetric,  the jumping measure $J$ of $(\sE,\sF)$ can be defined using the method outlined in \cite[(3.2.7)]{FOT11}; see also  \cite[Theorem~2.6]{M99}.  Furthermore,
\begin{equation}\label{eq:38}
	\sE(f,g)=-2\int_{E\times E\setminus d}f(x)g(y)J(dxdy)
\end{equation}
holds for any $f,g\in \sF\cap C_c(E)$ with disjoint support;  {see \cite[Theorem~2.9]{HMS06}}.   On the other hand,  following the argument in \cite[Lemma~5.3.2]{FOT11},  it can be verified that  \eqref{eq:38} also holds true with the canonical measure $\nu$ of $X$ in place of $J$.  Thus,  $\nu$ is also identical to $J$.  If $(\sE,\sF)$ is only quasi-regular,  a similar conclusion can be obtained by the transfer method using quasi-homeomorphism.  Particularly,  it holds that
\begin{equation}\label{eq:39}
	\int_{E\times E\setminus d}\left(\tilde{f}(x)-\tilde{f}(y) \right)^2\nu(dxdy)<\infty
\end{equation}
for any $f\in \sF$.
\end{remark}

A $\sigma$-finite positive measure $\mu$ on $E$ is called a \emph{smooth measure} (with respect to $\sE$) if it charges no $\sE$-polar sets and there exists an $\sE$-nest $\{K_n:n\geq 1\}$ consisting of compact sets such that $\mu(K_n)<\infty$ for all $n\geq 1$.  
Let $\sigma$ be a positive measure on $E\times E$,  and let $\sigma_l$ and $\sigma_r$ denote its left and right marginal measures on $E$, respectively.  Then, $\sigma$ is called a  \emph{bivariate smooth measure} (with respect to $\sE$) if 
\[
	\bar{\sigma}:=\frac{1}{2}\left(\sigma_l+\sigma_r\right)
\]
is smooth and $\sigma|_{E\times E\setminus d}\leq \nu$,  where $\nu$ is the canonical measure of $X$.  

When $M\in \text{MF}_+(X)$,   both $\rho_M$ and $\lambda_M$ are smooth measures with respect to $\sE$; see the proof of \cite[Theorem~3.4]{Y96}.  (This proof demonstrates the existence of two q.e. strictly positive and $\sE$-quasi-continuous functions $f$ and $g$ such that $\rho_M(f),\lambda_M(g)<\infty$.  Taking an $\sE$-nest consisting of compact sets $\{K_n\}$ such that $f,g\in C(\{K_n\})$ and $f,g>0$ on $\bigcup_{n\geq 1}K_n$, it can be readily verified  that $\rho_M(K_n),\lambda_M(K_n)<\infty$ for all $n\geq 1$.) According to \eqref{eq:46},  the bivariate Revuz measure $\nu_M$ is also bivariate smooth.  The converse also holds true; that is, the following statement is valid.

\begin{proposition}\label{PRO39}
The measure $\sigma$ on $E\times E$ is a bivariate smooth measure if and only if there exists $M\in \text{MF}_+(X)$ such that $\sigma=\nu_M$,  the bivariate Revuz measure of $M$.
\end{proposition}
\begin{proof}
See \cite[Theorem~4.3]{Y96}.  (The proof relies on the conclusion of \cite[Theorem3.5]{Y96}, which necessitates certain improvements in its own proof, as demonstrated in Theorem~\ref{THM52} of this article.)
\end{proof}
\begin{remark}\label{RM311}
It should be noted that if $(\sE,\sF)$ is local (see \cite[V, Definition~1.1]{MR92}),  meaning that $\nu=0$,  then every bivariate smooth/Revuz measure vanishes off the diagonal $d$. In this case,  killing by $M\in\text{MF}_+(X)$ reduces to perturbation by continuous additive functionals,  as examined in,  e.g.,  \cite[IV, \S4c]{MR92}.  In this context, we have 
\[
	\sF^M=\sF\cap L^2(E,\bar{\nu}_M),\quad \sE^M(f,g)=\sE(f,g)+\int_E \tilde{f}\tilde{g}d\bar{\nu}_M,\;f,g\in \sF^M,
\]
 where $\bar{\nu}_M:=\frac{1}{2}\left(\rho_M+\lambda_M \right)$.  
\end{remark}

\subsection{Killing by general MF}

Let us turn to consider the killing transformation induced by a general $M\in \text{MF}(X)$,  and we denote $G:=E_M$ for convenience. 

In light of Lemma~\ref{LM36} and Remark~\ref{RM37}~(2),  the killing transformation induced by $M \in \text{MF}(X)$ can be interpreted as first applying a killing by $D_{E\setminus G}$ to obtain the part process $X^{G}$,  as discussed in Section~\ref{SEC42}.  This is followed by a second killing on $X^{G}$ at the terminal time $S_M$ ($>0$,  a.s.),  and finally a killing transformation using $M\in \text{MF}_{++}(X^{G,S_M})$, where $X^{G,S_M}$ represents the subprocss of $X^{G}$ that is killed by $S_M$.  Note that  $X^G$ is properly associated with the part Dirichlet form $(\sE^G,\sF^G)$,  though it is not necessarily a Borel right process.  However,  the representation of the bivariate Revuz measure presented in \eqref{eq:46} depends on the assumption that the process is a Borel right process. Therefore, we cannot directly translate the characterization of general MF into the characterization of MF in  $\text{MF}_+(X^G)$  as described in \eqref{eq:46}.  This issue was inadvertently overlooked in \cite{Y96}, and we will address it in this subsection to rectify the oversight.  

\begin{lemma}\label{LM45}
Let $M\in \text{MF}(X)$ and $G=E_M$.  The marginal measures $\rho_M$ and $\lambda_M$ do not charge either $\sE$-polar sets or {$\sE^G$-polar subsets of $G$}.
\end{lemma}
\begin{proof}
By {Corollary~\ref{COR33}~(1)} and Lemma~\ref{LM32}~(1),  it suffices to demonstrate that $\rho_M(N)=\lambda_M(N)=0$ for any $m$-inessential Borel set $N\subset E$ for $X$.  For any $t>0$,   we have
\[
\bP_m\int_{s\in [0,t]\cap [0,\zeta)}1_{N}(X_s)d(-M_s)\leq \bP_m\left(X_s\in N, \exists s\in [0,\zeta)\right)=\bP_m(T_N<\infty)=0.
\]
Thus,  we conclude that $\rho_M(N)=0$.  Additionally,  according to \cite[(11.1)]{GS84},  $N$ is also \emph{left $m$-polar} in the sense that $\bP_m(S_N<\infty)=0$,  where $S_N:=\inf\{0<t<\zeta: X_{t-}\in N\}$.  Consequently, a similar argument establishes that $\lambda_M(N)=0$. 
\end{proof}
\begin{remark}
For $M\in \text{MF}(X)$,  the restriction of $\nu_M$ to $G\times G$ is clearly a bivariate smooth measure with respect to $\sE^{G}$.  (The canonical measure of $X^{G}$ is  $1_{G\times G}\cdot \nu$; see \cite[I, Theorem~3.9]{Y962}.) According to Proposition~\ref{PRO39},  this measure serves as the bivariate Revuz measure of an MF in $\text{MF}_+(X^{G})$.  By Lemma~\ref{LM36},  this MF corresponds precisely to $M$ itself.  
\end{remark}

We can now present the general form of Theorem~\ref{THM43}.  This result was originally examined in \cite[Theorem 4.1]{Y962};  however, its proof lacks rigor.  It is noteworthy that $\rho_M(E\setminus E_M)=\lambda_M(E\setminus E_M)=0$,  since $S_M\leq D_{E\setminus E_M}$.


\begin{theorem}\label{THM33}
Let $M\in \text{MF}(X)$ and $m^*:=m|_{E_M}$.   Then, the subprocess $X^M$ is a nearly $m^*$-symmetric right process on $E_M$, which is properly associated with the Dirichlet form $(\sE^M,\sF^M)$ {on $L^2(E_M, m^*)$} given by
\begin{equation}\label{eq:34}
	\begin{aligned}
		&\sF^M=\sF^{E_M}\cap L^2(E_M, \rho_M+\lambda_M), \\
		&\sE^M(f,g)=\sE(f,g)+\nu_M(\tilde{f}\otimes \tilde{g}),\quad f,g\in \sF^M,
	\end{aligned}
\end{equation}
where $\sF^{E_M}$ is defined as in \eqref{eq:31} with $G=E_M$.
 Furthermore,  $X^M$ is $m^*$-tight,  and $(\sE^M,\sF^M)$ is quasi-regular on $L^2(E_M,m^*)$.  
\end{theorem}
\begin{proof}
Note that $X^M$ is a right process on the Radon space $(E_M,\sB^*(E_M))$.  
We assert that $X$ is $m^*$-tight; consequently, the quasi-regularity of its associated Dirichlet form follows from,  e.g.,  \cite[Theorem~3.22]{F01} (provided that the sector condition is further verified).    In light of Lemma~\ref{LM42},  we need to consider the case where $M\in \text{MF}_+(X)$, where $E_M=E$ and $m^*=m$.
Let $\{K_n:n\geq 1\}$ be an increasing sequence of compact sets as specified in \eqref{eq:32}.  Clearly, for any $t\geq 0$,  we have
\[
	\left\{\omega\in \Omega: \lim_{n\rightarrow \infty}T_{E\setminus K_n}(k_t\omega)<\zeta(k_t\omega)\right\}\subset \left\{\omega\in \Omega: \lim_{n\rightarrow \infty}T_{E\setminus K_n}(\omega)<\zeta(\omega)\right\}.
\]
Therefore,  it follows from \eqref{eq:33} that
\[
\begin{aligned}
	\bQ_m \left(\lim_{n\rightarrow \infty}T_{E\setminus K_n}<\zeta \right)&=\bP_m \int_{(0,\infty]}1_{\left\{\lim_{n\rightarrow \infty}T_{E\setminus K_n}<\zeta\right\}}\circ k_t d(-M_t) \\
	&\leq \bP_m \left(\lim_{n\rightarrow \infty}T_{E\setminus K_n}<\zeta \right)=0.
\end{aligned}\]
Note that \eqref{eq:33} can be applied to this event because it remains valid for all $Z\in \mathrm{b}\cF^*$,  where  $\cF^*$ is the universal completion of $\cF^0$ (see \cite[Page 286]{S88}).  This establishes the $m^*$-tightness of $X^M$.

It remains to demonstrate that $X^M$ is nearly $m^*$-symmetric and to obtain the expression \eqref{eq:34} for its associated Dirichlet form $(\sE^M,\sF^M)$.  

Consider first that $M\in \text{MF}_+(X)$,  which allows for the decomposition given in \eqref{eq:410}.  Let $X^{\Gamma}$ be the subprocess of $X$ that is killed by the terminal time $J_{\Gamma}$ (namely, killed by the MF $1_{[0,J_\Gamma)}(t)$).  As demonstrated in \cite[Theorem~4.3]{Y962},  $X^\Gamma$ is $m$-nearly symmetric on $E$,  and \eqref{eq:34} holds with the bivariate Revuz measure given by
\begin{equation}\label{eq:412}
	\nu_\Gamma(F)=\lim_{t\downarrow 0}\frac{1}{t}\bP_m\left( F(X_{J_\Gamma-},X_{J_\Gamma}); J_\Gamma\leq t\right),\quad F\in \mathrm{p}\sB(E\times E).
\end{equation}
By applying \cite[Corollary~3.23]{F01} to $X^\Gamma$,  we can select an $m$-inessential Borel set $N$ for $X^\Gamma$ such that the restricted Borel right process $X^1:=X^\Gamma|_{E\setminus N}$ is associated with the same Dirichlet form as that of $X^\Gamma$.  Notably,  $N$ is $\sE$-polar by \cite[Corollary~3.3]{Y96}.  Consequently, it follows from Lemma~\ref{LM45} that 
\begin{equation}\label{eq:411}
\rho_\Gamma(N)=\lambda_\Gamma(N)=\rho_M(N)=\lambda_M(N)=0,
\end{equation}
where $\rho_\Gamma$ and $\lambda_\Gamma$ represent the right and left marginal measures of $\nu_\Gamma$, respectively.  {Write $E_1:=E\setminus N$ for convenience.}
According to Remark~\ref{RM37}~(2),  $M\in \text{MF}_{++}(X^\Gamma)$,  which clearly indicates that $M\in \text{MF}_{++}(X^1)$.   We then have
\[
	Q^1_t(f|_{E_1})(x):=Q_tf(x)=\bP_x \left(f(X_t)M_t\right)=\bP_x\left(f(X^1_t)M_t\right),\quad f\in \mathrm{p}\sB(E), x\in E_1.
\]
By applying Theorem~\ref{THM43} to $X^1$ and $M$,  we can conclude that $Q^1_t$ (and hence $Q_t$) acts on $L^2(E,m)$ as a strongly continuous contraction semigroup whose infinitesimal generator satisfies the sector condition.  The Dirichlet form of the subprocess $(X^1,M)$ is evidently the same as that of $X^M$.  By Theorem~\ref{THM43},  it is given by
\[
\begin{aligned}
	&\sF^M=\sF\cap L^2(E,\rho_\Gamma+\lambda_\Gamma+\rho_{M}^\Gamma+ \lambda_{M}^\Gamma),\\
	&\sE^M(f,g)=\sE(f,g)+\nu_\Gamma(\tilde{f}\otimes \tilde{g})+\nu^\Gamma_{M}(\tilde{f}\otimes \tilde{g}),\quad f,g\in \sF,
\end{aligned}\]
where $\nu^\Gamma_{M}$ is the bivariate Revuz measure of $M$ with respect to $X^1$.  To finalize \eqref{eq:34},  it remains to establish
\begin{equation}\label{eq:413}
	\nu_M(F)=\nu_\Gamma(F)+\nu^\Gamma_{M}(F),\quad F\in \mathrm{p}\sB(E_1\times E_1). 
\end{equation}
(Any function in $\mathrm{p}\sB(E_1\times E_1)$ can be treated as a function in $\mathrm{p}\sB(E\times E)$ through zero extension, as we have shown \eqref{eq:411}.) In fact,  we have
\[
\begin{aligned}
	\nu^\Gamma_{M}(F)&=\lim_{t\downarrow 0}\frac{1}{t}\bP_m \int_0^t F(X^1_{s-},X^1_s)d(-M_s) \\
	&=\lim_{t\downarrow 0}\frac{1}{t}\bP_m \int_{s\in [0,t]\cap (0,\zeta\wedge J_\Gamma)} F(X_{s-},X_s)d(-M_s).
\end{aligned}\]
From \eqref{eq:412},  we can deduce that
\[
	\nu^\Gamma_{M}(F)+\nu_\Gamma(F)=\lim_{t\downarrow 0}\frac{1}{t}\bP_m \int_{s\in [0,t]\cap (0,\zeta)} F(X_{s-},X_s)d(-M_s)=\nu_M(F).
\]
Therefore,  \eqref{eq:413} is established. 

Finally,  we consider the case where $M\in \text{MF}(X)$.  Let $G=E_M$,  which is finely open with respect to $X$, and let $B:=E\setminus G$.  {As noted in the proof of Lemma~\ref{LM32} (4), one may choose an $m$-inessential Borel set $N_0$ such that $G\setminus N_0 \in \sB(E\setminus N_0)$ is finely open with respect to the restricted process $X|_{E\setminus N_0}$. Furthermore, by \cite[Corollary 3.23]{F01}, we may further ensure that $X|_{E\setminus N_0}$ is a special Borel standard process. Passing to $E\setminus N_0$ entails no loss of generality, and none of the subsequent arguments or conclusions are affected by this reduction. Thus},  we may assume that $X$ is a special,  Borel standard process on $E$ and $G\in \sB(E)$.  Let $X^G$ represent the part process of $X$ on $G$ associated with the part Dirichlet form $(\sE^G,\sF^G)$,  and let $N$ be an $m^*$-inessential Borel set $N\subset G$ such that the restricted process $X^1:=X^G|_{G\setminus N}$ is a Borel standard process on $G\setminus N$.  It follows from Lemma~\ref{LM45} that
\begin{equation}\label{eq:414}
	\rho_M(N)=\lambda_M(N)=0.
\end{equation}
Note that $M\in \text{MF}_+(X^G)$ by Lemma~\ref{LM36},  whose restriction to $G\setminus N$ is an MF in $\text{MF}_+(X^1)$.  The transition semigroup of the subprocess $(X^1,M)$ is given by
\[
	Q^1_t (f|_{G\setminus N})(x):=\bP_x\left(f(X^1_t)M_t \right)=\bP_x\left(f(X_t)M_t \right)=Q_tf(x)
\]
for $x\in G\setminus N$ and $f\in \sB(G)$.  Based on our examination in the previous case,  we conclude that $Q^1_t$ (and hence $Q_t$) acts on $L^2(G,m^*)$ as a strongly continuous contraction semigroup, whose infinitesimal generator satisfies the sector condition.  The Dirichlet form of $X^M$ is identical to that of $(X^1,M)$,  which has the following form:
\[
\begin{aligned}
	&\sF=\sF^G\cap L^2(G\setminus N, \rho^1_{M}+\lambda^1_{M}),  \\
	&\sE(f,g)=\sE^G(f,g)+\nu^1_{M}(\tilde{f}\otimes \tilde{g}),\quad f,g\in \sF,
\end{aligned}\]
where $\nu^1_{M}$ is the bivariate Revuz measure of $M$ with respect to $X^1$,  and $\rho^1_{M}$ and $\lambda^1_{M}$ represent the right and left marginal measures of $\nu^1_{M}$, respectively.  Given \eqref{eq:414}, it remains to show that
\[
\nu_M(F)=\nu^1_{M}(F),\quad F\in\mathrm{p}\sB((G\setminus N)\times (G\setminus N)),
\]
where $F$ can be regarded as a function on $E\times E$ via zero extension.
In fact,  since $X_{D_B}\in B$ whenever $D_B<\infty$ and $S_M\leq D_B$,  we have
\[
\begin{aligned}
	\nu^1_{M}(F)&=\lim_{t\downarrow 0}\frac{1}{t}\bP_{m^*} \int_0^t F(X^1_{s-},X^1_s)d(-M_s) \\
	&=\lim_{t\downarrow 0}\frac{1}{t}\bP_{m^*} \int_{s\in [0,t]\cap (0,D_B\wedge \zeta)} F(X_{s-},X_s)d(-M_s) \\
	&=\lim_{t\downarrow 0}\frac{1}{t}\bP_{m} \int_{s\in [0,t]\cap (0,D_B]\cap (0, \zeta)} F(X_{s-},X_s)d(-M_s) \\
	&=\nu_M(F).
\end{aligned}\]
This completes the proof.
\end{proof}
\begin{remark}\label{RM49}
As explained before {Corollary~\ref{COR33}~(1)}, the $\sE^G$-quasi-notion can be inherited from the $\sE$-quasi-notion. We will now demonstrate that the $\sE^M$-quasi-notion and the $\sE^G$-quasi-notion are equivalent. Without loss of generality, assume that $M \in \text{MF}_+(X)$.  Let
\[
	\sF^1:=\sF\cap L^2(E,\rho_M+\lambda_M),\quad \sE^1(f,g):=\sE(f,g)+\int_E \tilde{f}\tilde{g}d(\rho_M+\lambda_M).
\]
As shown in \cite[Lemma~5.1]{CF12},  $(\sE^1,\sF^1)$ is a Dirichlet form on $L^2(E,m)$.  (Recall that $\rho_M+\lambda_M$ charges no $\sE$-polar sets by Lemma~\ref{LM45}.)  According to \cite[Theorem~5.1.4]{CF12},  the $\sE^1$-nest is equivalent to the $\sE$-nest.  (Although only symmetric Dirichlet forms are considered in \cite{CF12}, the $\sE$-quasi-notion pertains solely to the symmetric part of $\sE$, thereby making the relevant discussions applicable to non-symmetric Dirichlet forms as well.  See also \cite[Proposition~2.3]{RS95} for the examination of non-symmetric case.) Note that
\[
	\sF^M=\sF^1, \quad  \sE^M(f,f)\leq \sE^1(f,f),\; \forall f\in \sF^M.
\]
Thus, any $\sE^1$-nest is also an $\sE^M$-nest.  Conversely, it can be easily verified that $(\sE^M,\sF^M)$ is strongly subordinate to $(\sE,\sF)$. (To demonstrate the denseness of $\sF^M$ in $\sF$,  it suffices to consider the $\sE$-nest $\{K_n\}$ consisting of compact sets for which $\rho_M(K_n)+\lambda_M(K_n)<\infty$ for each $n\geq 1$, and note that $\bigcup_{n\geq 1}\mathrm{b}\sF^M_{K_n}=\bigcup_{n\geq 1}\mathrm{b}\sF_{K_n}$ is $\sE_1$-dense in $\sF$.) By \cite[Corollary~3.3]{Y96},  any $\sE^M$-nest is also an $\sE$-nest.  Therefore,  the $\sE^M$-nest is equivalent to the $\sE$-nest. In particular,  $\sE^M$-quasi-continuity (resp.  $\sE^M$-polar set) is equivalent to $\sE$-quasi-continuity (resp. $\sE$-polar set).   
\end{remark}

In the special case where $(\sE,\sF)$ is symmetric,  it is evident that $(\sE^M,\sF^M)$ is symmetric if and only if $\nu_M$ is symmetric,  i.e., $\nu_M(dxdy)=\nu_M(dydx)$.  For further exploration of the symmetric case, see \cite{Y98}.

\subsection{Killing is domination}

From Theorem~\ref{THM33},  we can readily derive the following fact.

\begin{corollary}\label{COR34}
Let $M\in \text{MF}(X)$ such that $m(E\setminus E_M)=0$.  Then, the Dirichlet form $(\sE^M,\sF^M)$ associated with the subprocess $X^M$ is dominated by $(\sE,\sF)$.  Furthermore,  when considered as a Dirichlet form on $L^2(E,m)$,  $(\sE^M,\sF^M)$ remains quasi-regular. 
\end{corollary}
\begin{proof}
The domination property can be established by verifying the second condition of Lemma~\ref{LM21} using the expression in \eqref{eq:34} for $(\sE^M,\sF^M)$.  (This is also clear from the inequality $Q_tf\leq P_tf$ for all $f\in \mathrm{p}\sB(E)$).  We now proceed to demonstrate the quasi-regularity of $(\sE^M,\sF^M)$ on $L^2(E,m)$.  To maintain clarity, we denote  $(\sE^M,\sF^M)$ on $L^2(E,m)$ as $(\bar{\sE}^M,\bar{\sF}^M)$. It is straightforward to verify that each compact subset of $E_M$ is also compact in $E$  because $E_M$ is endowed with the relative topology from $E$.  Consequently,  an $\sE^M$-nest consisting of compact sets is also an $\bar{\sE}^M$-nest.  Therefore,  the quasi-regularity of $(\bar{\sE}^M,\bar{\sF}^M)$ is a consequence of the quasi-regularity of $(\sE^M,\sF^M)$.  
\end{proof}

We have previously demonstrated in Theorem~\ref{THM24} that there exists a unique ``intermediate" Dirichlet form $(\tilde\sE,\tilde\sF)$ such that $\tilde\sE$ dominates $\sE^M$, and $\sE$ is a Silverstein extension of $\tilde\sE$. According to Theorem~\ref{THM33},  this intermediate Dirichlet form corresponds precisely to the part process $X^{E_M}$.

\section{Domination is killing: Quasi-regular case}

In this section,  we aim to establish the converse of Corollary~\ref{COR34}.  Specifically, we will demonstrate that every domination of a Dirichlet form corresponds to the killing transformation induced by an MF $M\in \text{MF}(X)$ with $m(E\setminus E_M)=0$.  

Let $(\sE,\sF)$ and $(\sE',\sF')$ be two Dirichlet forms on $L^2(E,m)$ such that $\sE' \preceq \sE$.  Denote by $(\tilde{\sE},\tilde{\sF})$ the intermediate Dirichlet form obtained in Theorem~\ref{THM24}. 

\subsection{Strong subordination}

Let us first interpret the probabilistic meaning of the strong subordination of $(\sE',\sF')$ to $(\tilde{\sE},\tilde{\sF})$.  It should be pointed out that the Dirichlet form $(\sE,\sF)$ does not play a role in this subsection. 

The following fact is elementary,  and its proof is provided in \cite[Corollary~3.3]{Y96}.

\begin{lemma}\label{LM51}
Any $\sE'$-nest is also an $\tilde{\sE}$-nest.  In particular,  $(\tilde{\sE},\tilde{\sF})$ is quasi-regular provided that $(\sE',\sF')$ is quasi-regular.
\end{lemma}

The following result essentially establishes that strong subordination corresponds to a killing transformation induced by an MF in $\text{MF}_+$. While this characterization was discussed in \cite[Theorem~3.5]{Y96}, the proof provided therein appears incomplete. Below, we present a new proof.


\begin{theorem}\label{THM52}
Assume that $(\sE',\sF')$ is quasi-regular on $L^2(E,m)$. Then there exists a Borel right process  $\tilde{X}=(\tilde{X}_t)_{t\geq 0}$ properly associated with $(\tilde{\sE},\tilde{\sF})$ and an MF $M\in \text{MF}_+(\tilde{X})$ such that the subprocess $\tilde X^M$ of $\tilde X$, killed by $M$, is properly associated with $(\sE',\sF')$.  \end{theorem}
\begin{proof}
Let $d$ be a metric on $E$ that induces the topology of $E$, and let $C_d(E)$ denote the family of all $d$-uniformly continuous functions on $E$.  The set $\mathrm{b}C_d(E)$ consists of all bounded functions in $C_d(E)$. Since $E$ is Lusin,  there exists a countable subset $\sC\subset \mathrm{b}C_d(E)$ that is dense in $\mathrm{b}C_d(E)$ with respect to the uniform norm (see \cite[A.2]{S88}).  Define $\sC^+:=\{u^+:u\in \sC\}$,  where $u^+:=u\vee 0$.  As $m$ is $\sigma$-finite,  we can select a sequence of increasing Borel sets $E_n$ such that $\bigcup_{n\geq 1}E_n=E$ and $m(E_n)<\infty$.  Let
\[
	\sC^+_1:=\{u^+\cdot 1_{E_n}: u\in \sC, n\geq 1\}. 
\]
Clearly,  $\sC^+_1\subset \mathrm{b}\sB(E)\cap L^2(E,m)$ is countable.  Denote by $\mathbb{Q}_+$ the set of all positive rationals.  

Since $(\tilde{\sE},\tilde{\sF})$ is quasi-regular by Lemma~\ref{LM51},  there exists a Borel right process $\tilde{X}^1$ properly associated with $(\tilde{\sE},\tilde{\sF})$.  Namely,  for any $f\in \sB(E)\cap L^2(E,m)$ and $\alpha>0$,  $\tilde{U}^1_\alpha f$ is an $\tilde{\sE}$-quasi-continuous $m$-version of $\tilde{G}_\alpha f$,  where $\tilde{U}^1_\alpha$ is the resolvent of $\tilde{X}^1$, and $\tilde{G}_\alpha$ is the $L^2$-resolvent of $(\tilde{\sE},\tilde{\sF})$.  See \cite[IV, Proposition~2.8]{MR92}.  Similarly, there exists a Borel right process $Y$ such that its resolvent $V^1_\alpha f$ is an $\sE'$-quasi-continuous $m$-version of $G'_\alpha f$ for any $f\in \mathrm{b}\sB(E)\cap L^2(E,m)$ and $\alpha>0$, where $G'_\alpha$ is the $L^2$-resolvent of $(\sE',\sF')$.  

Since strong subordination implies domination,  it follows from Lemma~\ref{LM21}~(1) that
\begin{equation}\label{eq:51}
	V^1_\alpha f\leq \tilde{U}^1_\alpha f,\quad m\text{-a.e.} 
\end{equation}
for all positive $f\in \mathrm{b}\sB(E)\cap L^2(E,m)$ and $\alpha>0$. Note that $V^1_\alpha f$ is $\tilde{\sE}$-quasi-continuous by Lemma~\ref{LM51}.  Thus,  there exists an $\tilde{\sE}$-nest $\{F_n:n\geq 1\}$ such that 
\[
	\tilde{U}^1_\alpha f,V^1_\alpha f\in C(\{F_n\}),\quad \forall f\in \sC^+_1,\alpha \in \mathbb{Q}_+.
\]
In particular,  it follows from \eqref{eq:51} that 
\begin{equation}\label{eq:53}
	V^1_\alpha f(x)\leq \tilde{U}^1_\alpha f(x),\quad \forall x\in \bigcup_{n\geq 1}F_n, f\in \sC^+_1,\alpha\in \mathbb{Q}_+. 
\end{equation}

Let us now construct an increasing sequence of $m$-inessential Borel sets $\{N_k:k\geq 1\}$ for $\tilde{X}^1$.  We begin by selecting an $m$-inessential Borel set $N_1$ for $\tilde{X}^1$ that contains $E\setminus \bigcup_{n\geq 1}F_n$.  Since $V^1_\alpha 1_{N_1}=0$, $m$-a.e., and $V^1_\alpha 1_{N_1}$ is $\sE'$-quasi-continuous,  for all $\alpha \in \mathbb{Q}_+$,  there exists an $\sE'$-polar set $N'_1$ such that 
\[
	V^1_\alpha 1_{N_1}(x)=0,\quad \forall x\in E\setminus N'_1,\alpha\in \mathbb{Q}_+.
\]
Note that $N'_1$ is also $\tilde{\sE}$-polar by Lemma~\ref{LM51}.  Thus, we can select an $m$-inessential Borel set $N_2$ for $\tilde{X}^1$ such that $N_1\cup N'_1\subset N_2$.  For $k\geq 2$, once $N'_{k-1}$ and $N_k$ have been defined,  we construct $N'_k$ and $N_{k+1}$ using the following method: Select an $\sE'$-polar set $N'_{k}$ that contains $N'_{k-1}$ such that
\begin{equation}\label{eq:52}
	V^1_\alpha 1_{N_k}(x)=0,\quad \forall x\in E\setminus N'_{k},\alpha\in \mathbb{Q}_+,
\end{equation}
and then choose an $m$-inessential Borel set $N_{k+1}$ for $\tilde{X}^1$ such that $N_k\cup N'_k\subset N_{k+1}$.  

Define $N:=\bigcup_{k\geq 1}N_k$ and $E_0:=E\setminus N$.  It is evident that $E_0$ is a Borel invariant set for $\tilde{X}^1$. Hence,  the restricted process $\tilde{X}^1|_{E_0}$ remains a Borel right process.  Let $\tilde{X}$ be the trivial extension of $\tilde{X}^1|_{E_0}$ to $E$,  namely,  $\tilde{X}$ will remain at $x$ indefinitely if $\tilde{X}_0=x\in N$.  Clearly,  $\tilde{X}$ is also a Borel right process properly associated with $(\tilde{\sE},\tilde{\sF})$.  Denote by $(\tilde U_\alpha)_{\alpha>0}$ the resolvent of $\tilde{X}$.  Note that for all $f\in \mathrm{b}\sB(E)$ and $\alpha>0$,
\[
	\tilde{U}_\alpha f=\tilde{U}^1_\alpha f \cdot 1_{E_0}+\frac{1}{\alpha}f\cdot 1_{N}.
\]  
On the other hand,  it follows from \eqref{eq:52} that
\[
	V^1_\alpha 1_N(x)=0,\quad \forall x\in E_0,\alpha>0.
\]
Particularly,  
\begin{equation}\label{eq:55}
	V_\alpha f:=V_\alpha^1f\cdot 1_{E_0}+\frac{1}{\alpha}f\cdot 1_{N},\quad \forall f\in \mathrm{b}\sB(E),\alpha>0
\end{equation}
provides a well-defined resolvent on $E$.  

We assert that 
\begin{equation}\label{eq:54}
	V_\alpha f(x)\leq \tilde{U}_\alpha f(x),\quad \forall x\in E,\alpha>0, f\in \mathrm{bp}\sB^*(E).
\end{equation}
It suffices to consider (and fix) $x\in E_0$.  From \eqref{eq:53}, we know that \eqref{eq:54} holds for $f\in \sC^+_1$ and $\alpha\in \mathbb{Q}_+$.  By applying the monotone convergence theorem (for $1_{E_n}\uparrow 1_E$) and the dominated convergence theorem (for the parameters $\mathbb{Q}_+\ni\alpha_n\rightarrow \alpha>0$),  we can extend \eqref{eq:54} to all $f\in \sC^+$ and all $\alpha>0$.    Fix $\alpha>0$.  Since $\sC$ is dense in $\mathrm{b}C_d(E)$ and $|u^+(y)-v^+(y)|\leq |u(y)-v(y)|$ for all $u,v\in \mathrm{b}C_d(E)$ and $y\in E$,  it follows that \eqref{eq:54} holds for all positive $f\in \mathrm{b}C_d(E)$.  Then, by utilizing \cite[Proposition~A2.1]{S88},  \eqref{eq:54} holds for all positive and bounded $f\in C(E)$.  Note that the Lusin topological space $E$ is normal. Therefore,  Urysohn's lemma (see,  e.g., \cite[Lemma~4.15]{F99}) indicates that \eqref{eq:54} holds for $f=1_K$,  where $K$ is an arbitrary compact set.  Since the finite measure $\mu(\cdot):=V_\alpha(x,\cdot)+\tilde{U}_\alpha(x,\cdot)$ is inner regular on $E$ (see \cite[Theorem~A2.3]{S88}),  we can further obtain \eqref{eq:54} with $f=1_B$ for any Borel set $B$.  This result can be extended to all $f\in \mathrm{b}\mathrm{p}\sB(E)$ by \cite[Theorem~2.10]{F99} and the monotone convergence theorem.  For any $f\in \mathrm{bp}\sB^*(E)$,  there exist $f_1,f_2\in \mathrm{b}\mathrm{p}\sB(E)$ such that $f_1\leq f\leq  f_2$ and $\mu(f_2-f_1)=0$.  Therefore, we can eventually conclude \eqref{eq:54} for all $f\in  \mathrm{bp}\sB^*(E)$. 

For any $x\in E_0$,  it is straightforward to see that
\[
	\alpha V_\alpha 1_E(x)=\alpha V^1_\alpha 1_E(x)=\bE_x \int_0^\infty e^{-t}1_E(Y_{t/\alpha})dt\rightarrow 1_E(x),\quad \alpha\rightarrow \infty.
\]
Evidently,  $\alpha V_\alpha 1_E(x)=1_E(x)$ for all $x\in N$.  Hence,  it follows from \cite[III, Theorem~4.9]{BG68} that $V_\alpha$ is \emph{exactly subordinate} to $U_\alpha$ in the sense of \cite[III, Definition~4.8]{BG68}.  In view of \cite[Theorem~(56.14)]{S88},  there exists a (unique) exact MF $M=(M_t)_{t\geq 0}$ of $\tilde{X}$ such that
\[
	V_\alpha f(x)=\bE_x \int_0^\infty e^{-\alpha t}f(\tilde{X}_t)M_t dt,\quad \forall x\in E, f\in \mathrm{p}\sB^*(E).
\]
It is easy to verify that $E_M=E$.  (If $\bP_x(M_0=0)=1$,  then $V_\alpha f(x)=0$ for all $\alpha>0$ and $f\in \mathrm{b}\sB^*(E)\supset \mathrm{b}C(E)$.  We must have $x\in E_0$ and $f(x)=\lim_{\alpha\rightarrow \infty}\alpha V^1_\alpha f(x)=\lim_{\alpha\rightarrow \infty}\alpha V_\alpha f(x)=0$ for all $f\in \mathrm{b}C(E)$,  which leads to a contradiction.) Particularly,  $M\in \text{MF}_+(\tilde{X})$ and $V_\alpha$ is the resolvent of the subprocess $\tilde{X}^M$.  

By utilizing \eqref{eq:55},  it is straightforward to verify that $\tilde{X}^M$ is $m$-nearly symmetric and that its Dirichlet form coincides with that of $Y$.  In other words,  the Dirichlet form of $\tilde{X}^M$ is $(\sE',\sF')$.  By Theorem~\ref{THM33}, we can conclude that $\tilde{X}^M$ is properly associated with $(\sE',\sF')$.  This completes the proof.
\end{proof}

From the aforementioned theorem and Remark~\ref{RM49}, it can be concluded that strong subordination guarantees the equivalence between the quasi-notions of the two Dirichlet forms.

\begin{corollary}\label{COR53}
An increasing sequence of closed sets is an $\tilde{\sE}$-nest  if and only if it is an $\sE'$-nest. Particularly, $N\subset E$ is an $\tilde{\sE}$-polar set if and only if it is an $\sE'$-polar set.  A function $f$ is $\tilde{\sE}$-quasi-continuous if and only if it is $\sE'$-quasi-continuous.
\end{corollary}

\subsection{Domination is killing}

We are now in a position to state the converse of Corollary~\ref{COR34} under the assumption that  $(\sE,\sF)$ is quasi-regular.  

\begin{theorem}\label{THM41}
Let $(\sE,\sF)$ and $(\sE',\sF')$ be two quasi-regular Dirichlet forms on $L^2(E,m)$ such that $\sE'\preceq \sE$, and let $(\tilde{\sE},\tilde{\sF})$ denote the unique Dirichlet form obtained in Theorem~\ref{THM24}.  Then, there exists a Borel right process $X$ properly associated with $(\sE,\sF)$,  a finely open {Borel set $G$} with respect to $X$ satisfying $m(E\setminus G)=0$,  and an MF $M\in \text{MF}_+(X^G)$, where $X^G$ is the part process of $X$ on $G$, such that
\begin{itemize}
\item[(1)] $X^G$ is properly associated with $(\tilde{\sE},\tilde{\sF})$.
\item[(2)] The subprocess of $X^G$ killed by $M$ is properly associated with $(\sE',\sF')$.  
\end{itemize}
\end{theorem}
\begin{proof}
According to Theorem~\ref{THM24},  $(\sE,\sF)$ is a Silverstein extension of $(\tilde{\sE},\tilde{\sF})$.  It follows that $\tilde{\sF}$ is a closed order ideal in $\sF$.  Therefore,  the argument presented by Stollman \cite{S93} demonstrates that there exists an $\sE$-quasi-open set $G_1$ such that  $(\tilde{\sE},\tilde{\sF})$ is the part Dirichlet form of $(\sE,\sF)$ on $G_1$;  see also \cite[Remark~5.13]{F01}.  It is evident that $m(E\setminus G_1)=0$,  as $(\tilde\sE,\tilde\sF)$ is a Dirichlet form on $L^2(E,m)$.  

By Lemma~\ref{LM32}~(4) (and its proof),  we can take a Borel right process $X$ properly associated with $(\sE,\sF)$ and assume, without loss of generality, that $G_1$ is a Borel, finely open set with respect to $X$.  It is straightforward to verify that if $N\subset G_1$ is an $m|_{G_1}$-inessential Borel set for the part process $X^{G_1}$,  then $G_1\setminus N$ is also finely open with respect to $X$, and the restricted right process $X^{G_1}|_{G_1\setminus N}$ is precisely the part process $X^{G_1\setminus N}$ of $X$ on $G_1\setminus N$.  Since $X^{G_1}$ is properly associated with the quasi-regular Dirichlet form $(\sE^{G_1},\sF^{G_1})$, as investigated in Lemma~\ref{LM42},  there exists an $m|_{G_1}$-inessential Borel set $N_1$ for $X^{G_1}$ such that $X^{G_1}|_{G_1\setminus N_1}=X^{G_1\setminus N_1}$ is a Borel right process.  
Define $G_2:=G_1\setminus N_1$.  This set is finely open with respect to $X$, and $X^{G_2}$ is properly associated with $(\sE^{G_2},\sF^{G_2})$.  Clearly,  $X^{G_2}$ is also properly associated with $(\sE^{G_1},\sF^{G_1})=(\tilde{\sE},\tilde{\sF})$ since any $\sE^{G_2}$-nest consisting of compact sets is also an $\sE^{G_1}$-nest.   

Recall that $(\sE',\sF')$ is strongly subordinate to $(\tilde{\sE},\tilde{\sF})=(\sE^{G_2},\sF^{G_2})$.  It follows from Corollary~\ref{COR53} that the $\tilde{\sE}$-polar set $E\setminus G_2$ is also $\sE'$-polar.  Therefore,  it is straightforward to verify that $(\sE',\sF')$ is quasi-regular on $L^2(G_2,m|_{G_2})$.  By following the argument presented after \eqref{eq:51} in the proof of Theorem~\ref{THM52} (considering the restrictions of $V_\alpha f, \tilde{U}_\alpha f$ to $E_0$ in \eqref{eq:54} instead) for $(\sE',\sF')$ and $(\tilde{\sE},\tilde{\sF})$ on $L^2(G_2,m|_{G_2})$,  we can select an $m|_{G_2}$-inessential Borel set $N_2$ for $X^{G_2}$ and an MF $M\in \text{MF}_+(X^{G_2}|_{G_2\setminus N_2})$ such that the subprocess of $X^{G_2}|_{G_2\setminus N_2}$ killed by $M$ is properly associated with $(\sE',\sF')$ on $L^2(G_2\setminus N_2, m|_{G_2\setminus N_2})$.  

Define $G:=G_2\setminus N_2$.  Then $G$ is finely open with respect to $X$,  $m(E\setminus G)=0$,  and $X^G=X^{G_2}|_{G_2\setminus N_2}$ is properly associated with $(\tilde{\sE},\tilde{\sF})$.  In addition,  $M\in \text{MF}_+(X^G)$,  and the subprocess of $X^G$ killed by $M$ is properly associated with $(\sE',\sF')$ on both $L^2(G,m|_G)$ and $L^2(E,m)$.  This completes the proof.
\end{proof}
\begin{remark}
According to this proof, the {finely open} set $G$ in the theorem can be selected as a Borel set, and the part process $X^G$ can likewise be taken as a Borel right process.
\end{remark}

The following corollary readily follows from Lemma~\ref{LM36}. 

\begin{corollary}\label{COR46}
Let $(\sE,\sF)$ and $(\sE',\sF')$ be two quasi-regular Dirichlet forms on $L^2(E,m)$ such that $\sE'\preceq \sE$.  Then, there exists a Borel right process $X$ properly associated with $(\sE,\sF)$ and $M\in \text{MF}(X)$ with $m(E\setminus E_M)=0$ such that the subprocess $X^M$ is properly associated with $(\sE',\sF')$.  
\end{corollary}

\section{Domination is killing: General case}\label{SEC5}

Now, we remove the assumption of quasi-regularity for $(\sE,\sF)$.  In this setting, a right process corresponding to $(\sE,\sF)$ on $E$ may not exist. However, by following the approach of Silverstein \cite{S74}, we can extend the space $E$ to establish the conclusions of Theorem~\ref{THM41} on the expanded space. See also \cite[Theorem~6.6.5]{CF12}. The main result is stated as follows.

\begin{theorem}\label{THM51}
Let $(\sE',\sF')$ be a quasi-regular Dirichlet form on $L^2(E,m)$, and let $(\sE,\sF)$ be another Dirichlet form on $L^2(E,m)$,  not necessarily quasi-regular, such that $\sE'\preceq \sE$.  Denote by $(\tilde{\sE},\tilde{\sF})$ the unique Dirichlet form obtained in Theorem~\ref{THM24}. Then, there exists a locally compact separable metric space $\widehat{E}$ and a measurable map $j:E\rightarrow \widehat{E}$ such that,  by defining $\widehat{m}:=m\circ j^{-1}$,  the operator $j^*$ is a unitary map from $L^2(\widehat{E},\widehat{m})$ onto $L^2(E,m)$,  and the following hold:
\begin{itemize}
\item[(i)] The image Dirichlet form $(\widehat{\sE},\widehat{\sF}):=j(\sE,\sF)$ is a regular Dirichlet form on $L^2(\widehat{E},\widehat{m})$;
\item[(ii)] There exists an $\widehat{\sE}$-quasi-open subset $\widehat{G}$ of $\widehat{E}$ with $\widehat{m}(\widehat{E}\setminus \widehat{G})=0$ such that $j$ is a quasi-homeomorphism that maps $(\tilde{\sE},\tilde{\sF})$ to $j(\tilde{\sE},\tilde{\sF})=(\widehat{\sE}^{\widehat{G}},\widehat{\sF}^{\widehat{G}})$,  the part Dirichlet form of $(\widehat{\sE},\widehat{\sF})$ on $\widehat{G}$;
\item[(iii)] $j$ is a quasi-homeomophism that maps $(\sE',\sF')$ to the quasi-regular image Dirichlet form $j(\sE',\sF')$ on $L^2(\widehat{E},\widehat{m})$.  
\end{itemize}
Furthermore,  there exists a Hunt process $\widehat{X}$ on $\widehat{E}$, with $\widehat{G}$ being taken as a finely open set with respect to $\widehat{X}$, and an MF $\widehat M\in \text{MF}_+(\widehat{X}^{\widehat{G}})$,  where $\widehat{X}^{\widehat{G}}$ is the part process of $\widehat{X}$ on $\widehat{G}$, such that $\widehat{X}$ is properly associated with $j(\sE,\sF)$,  $\widehat{X}^{\widehat{G}}$ is properly associated with $j(\tilde{\sE},\tilde{\sF})$, and the subprocess $(\widehat{X}^{\widehat{G}},\widehat M)$ is properly associated with $j(\sE',\sF')$.  
\end{theorem}
\begin{proof}
Note that $(\sE,\sF)$ is a Silverstein extension of $(\tilde{\sE},\tilde{\sF})$,  and according to Lemma~\ref{LM51},  $(\tilde{\sE},\tilde{\sF})$ is quasi-regular.  We can then apply an argument involving the Gelfand transformation, as in \cite[Theorem~6.6.5]{CF12}, to the symmetric parts of $(\tilde{\sE},\tilde{\sF})$ and $(\sE,\sF)$.  This yields an $\tilde{\sE}$-nest $\{F_n: n\geq 1\}$ consisting of compact sets,  a locally compact metrizable space $\widehat{E}$, and a Borel measurable map
\[
	j:\bigcup_{n\geq 1}F_n\rightarrow \widehat{E}
\]
such that $j|_{F_n}$ is a topological homeomorphism from $F_n$ to $\widehat{F}_n:=j(F_n)$ for each $n\geq 1$, $\widehat{m}=m\circ j^{-1}$ is a fully supported Radon measure on $\widehat{E}$, and $j^*$ is unitary from $L^2(\widehat{E},\widehat{m})$ onto $L^2(E,m)$.  Furthermore,  $(\widehat{\sE},\widehat{\sF}):=j(\sE,\sF)$ is a regular Dirichlet form on $L^2(\widehat{E},\widehat{m})$,  and there exists an $\widehat{\sE}$-quasi-open subset $\widehat{G}$ of $\widehat{E}$ with $\widehat{m}(\widehat{E}\setminus \widehat{G})=0$ such that $j$ is a quasi-homoemorphism that maps $(\tilde{\sE},\tilde{\sF})$ to the quasi-regular Dirichlet form $j(\tilde{\sE},\tilde{\sF})=(\widehat{\sE}^{\widehat{G}},\widehat{\sF}^{\widehat{G}})$ on $L^2(\widehat{E},\widehat{m})$.  In particular,  there exists another $\tilde{\sE}$-nest $\{K_n: n\geq 1\}$ such that $\{j(F_n\cap K_n):n\geq 1\}$ is an $\widehat{\sE}^{\widehat{G}}$-nest.  

Since $\sE'$ is strongly subordinate to $\tilde{\sE}$,  it follows from Corollay~\ref{COR53} that the $\tilde{\sE}$-nest $\{F'_n:=F_n\cap K_n: n\geq 1\}$ is also an $\sE'$-nest.  Clearly,  $j:F'_n\rightarrow j(F'_n)$ is a topological homoemorphism for each $n\geq 1$.  As $j^*$ is onto $L^2(E,m)$,  $(\widehat{\sE}',\widehat{\sF}'):=j(\sE',\sF')$ is well-defined as the image Dirichlet form on $L^2(\widehat{E},\widehat{m})$.  To verify that $j$ is a quasi-homeomorphism that maps $(\sE',\sF')$ to $(\widehat{\sE}',\widehat{\sF}')$,  it remains to show that $\{\widehat{F}'_n:=j(F'_n):n\geq 1\}$ forms an $\widehat{\sE}'$-nest.  In fact,  by the definition of the image Dirichlet form,  we have
\begin{equation}\label{eq:51-2}
\begin{aligned}
	\widehat{\sF}'_{\widehat{F}'_n}&=\{\widehat{f}\in L^2(\widehat{E},\widehat{m}): \widehat{f}\circ j\in \sF',\widehat{f}=0,\widehat{m}\text{-a.e. on }\widehat{E}\setminus j(F'_n)\} \\
	&=\{\widehat{f}\in L^2(\widehat{E},\widehat{m}): \widehat{f}\circ j\in \sF',\widehat{f}\circ j=0,m\text{-a.e. on }E\setminus F'_n\} \\
	&=\{\widehat{f}\in L^2(\widehat{E},\widehat{m}): \widehat{f}\circ j\in \sF'_{F'_n}\}.  
\end{aligned}\end{equation}
As a result,  $\bigcup_{n\geq 1}\widehat{\sF}'_{\widehat{F}'_n}$ is $\widehat{\sE}'_1$-dense in $\widehat{\sF}'=\{\widehat{f}\in L^2(\widehat{E},\widehat{m}): \widehat{f}\circ j\in \sF'\}$.  This confirms that $\{\widehat{F}'_n:n\geq 1\}$ is indeed an $\widehat{\sE}'$-nest.  

For the second part of the statements,  we observe that the existence of $\widehat{X}$ and the fact that $\widehat{G}$ can be chosen as a finely open set with respect to $\widehat{X}$ follow directly from \cite[Theorem~6.6.5]{CF12}.  It is straightforward to verify that $j(\sE',\sF')$ is strongly subordinate to $j(\tilde{\sE},\tilde{\sF})$ and that $j(
\sE,\sF)$ is a Silverstein extension of $j(\tilde{\sE},\tilde{\sF})$.  By applying Theorem~\ref{THM41} to these Dirichlet forms,  we can further identify an MF $M\in \text{MF}_+(\widehat{X}^{\widehat{G}})$ that satisfies the desired conditions.  This completes the proof.
\end{proof}

Similar to Corollary~\ref{COR46},  this result  directly leads to the following corollary.

\begin{corollary}\label{COR52}
Let $(\sE',\sF')$ be a quasi-regular Dirichlet form on $L^2(E,m)$, and let $(\sE,\sF)$ be another (not necessarily quasi-regular) Dirichlet form on $L^2(E,m)$ such that $\sE'\preceq \sE$.   Then there exists a locally compact separable metric space $\widehat{E}$ and a measurable map $j:E\rightarrow \widehat{E}$ such that, by defining $\widehat{m}:=m\circ j^{-1}$,  $j^*$ is a unitary map from $L^2(\widehat{E},\widehat{m})$ onto $L^2(E,m)$,  and the following hold:
\begin{itemize}
\item[(1)] The image Dirichlet form $j(\sE,\sF)$ is a regular Dirichlet form on $L^2(\widehat{E},\widehat{m})$,  and $j$ serves as a quasi-homeomorphism that maps $(\sE',\sF')$ to the quasi-regular image Dirichlet form $j(\sE',\sF')$ on $L^2(\widehat{E},\widehat{m})$.
\item[(2)] There exists a Hunt process $\widehat{X}$ on $\widehat{E}$ and an MF $\widehat{M}\in \text{MF}(\widehat{X})$ with $\widehat{m}(\widehat{E}\setminus E_{\widehat{M}})=0$ such that $\widehat{X}$ is properly associated with $j(\sE,\sF)$, and the subprocess $(\widehat{X},\widehat{M})$ is properly associated with $j(\sE',\sF')$.  
\end{itemize}
\end{corollary}
\begin{remark}
From Remark~\ref{RM311}, it is clear that if $\sE$ is local, then $\sE'$ is also local. This observation reflects the preservation of locality under the given conditions. For further discussions and related results in the regular and symmetric setting, see \cite[Theorem~4.3]{K18}.
\end{remark}

\section{Sandwiched Dirichlet form for domination}

We consider two Dirichlet forms $(\sE,\sF)$ and $(\sE',\sF')$ on $L^2(E,m)$ such that $\sE'\preceq \sE$.  The goal of this section is to explore the properties of a third Dirichlet form$(\sA,\cD(\sA))$ on $L^2(E,m)$ that lies between $(\sE, \sF)$ and $(\sE',\sF')$. Specifically,  $(\sA,\cD(\sA))$ satisfies the sandwiching property:
\[
\sE' \preceq \sA \preceq \sE.
\]
This setup enables a deeper investigation into the structure of intermediate forms and their associated processes.

\subsection{General characterization}

The primary assumption for this section is the quasi-regularity of $(\sE',\sF')$. However, to facilitate exposition, we will also assume the quasi-regularity of $(\sE,\sF)$ and $(\sA,\cD(\sA))$. While these additional assumptions streamline the analysis, they are not strictly necessary. Indeed, through arguments analogous to those in \S\ref{SEC5}, the results can be extended to the case where $(\sE,\sF)$ and $(\sA,\cD(\sA))$ are not quasi-regular. The following lemma illustrates this extension and demonstrates how quasi-regularity assumptions can be relaxed.

\begin{lemma}\label{LM61}
Let $(\sA,\cD(\sA))$ be a Dirichlet form on $L^2(E,m)$ satisfying $\sE'\preceq \sA\preceq \sE$,  where $(\sE',\sF')$ is assumed to be quasi-regular.  Then, there exists a locally compact separable metric space $\widehat{E}$ and a measurable map $j:E\rightarrow \widehat{E}$ such that,  by defining $\widehat{m}:=m\circ j^{-1}$,  the operator $j^*$ is a unitary map from $L^2(\widehat{E},\widehat{m})$ onto $L^2(E,m)$,  and the following hold:
\begin{itemize}
\item[(1)] $j$ is a quasi-homoemorphism that maps $(\sE',\sF')$ to the quasi-regular image Dirichlet form $j(\sE',\sF')$ on $L^2(\widehat{E},\widehat{m})$.
\item[(2)] The image Dirichlet form $j(\sA,\cD(\sA))$ is quasi-regular on $L^2(\widehat{E},\widehat{m})$.
\item[(3)] The image Dirichlet form $j(\sE,\sF)$ is regular on $L^2(\widehat{E},\widehat{m})$.   
\end{itemize}
Particularly,  the domination relationship
\[
	j(\sE',\sF')\preceq j(\sA,\cD(\sA))\preceq j(\sE,\sF)
\]	
is preserved under the transformation $j$.
\end{lemma}
\begin{proof}
We begin by applying Corollary~\ref{COR52} to the Dirichlet forms $(\sE',\sF')$ and $(\sA,\cD(\sA))$.  This provides us with a locally compact separable metric space $\widehat{E}_1$ and a measurable map $j_1:E\rightarrow \widehat{E}_1$ such that,  by letting $\widehat{m}_1:=m\circ j^{-1}_1$,  $j^*_1$ is a unitary map from $L^2(\widehat{E}_1,\widehat{m}_1)$ onto $L^2(E,m)$, $j_1$ acts as a quasi-homeomorphism that maps $(\sE',\sF')$ to the quasi-regular Dirichlet form $j_1(\sE',\sF')$,  and  $j_1(\sA,\cD(\sA))$ is regular on $L^2(\widehat{E}_1,\widehat{m}_1)$.  Since $j^*_1$ is onto $L^2(E,m)$,  it follows that $j_1(\sE,\sF)$ defines a Dirichlet form on $L^2(\widehat{E}_1,\widehat{m}_1)$.  Additionally,  the domination relationship
\[
	j_1(\sE',\sF')\preceq j_1(\sA,\cD(\sA))\preceq j_1(\sE,\sF)
\]	
is preserved under the transformation $j_1$. 
Next,  we apply Corollary~\ref{COR52} to $j_1(\sA,\cD(\sA))$ and $j_1(\sE,\sF)$.  This yields another locally compact separable metric space $\widehat{E}$ and a measurable map $j_2:\widehat{E}_1\rightarrow \widehat{E}$ such that, by letting $\widehat{m}:=\widehat{m}_1\circ j_2^{-1}$,  $j^*_2$ is a unitary map from $L^2(\widehat{E},\widehat{m})$ onto $L^2(\widehat{E}_1,\widehat{m}_1)$,  $j_2$ is a quasi-homeomorphism that maps $j_1(\sA,\cD(\sA))$ to the quasi-regular image Dirichlet form $j_2\circ j_1(\sA,\cD(\sA))$,  and the image Dirichlet form $j_2\circ j_1(\sE,\sF)$ is regular on $L^2(\widehat{E},\widehat{m})$.  Particularly,  $j_2\circ j_1(\sE',\sF')$ gives a Dirichlet form on $L^2(\widehat{E},\widehat{m})$.  

Define $j:=j_2\circ j_1$,  which is clearly Borel measurable from $E$ to $\widehat{E}$.  It is evident that $\widehat{m}=m\circ j^{-1}$,  $j^*$ is a unitary map from $L^2(\widehat{E},\widehat{m})$ onto $L^2(E,m)$,  and
\[
 j(\sA,\cD(\sA))=j_2\circ j_1(\sA,\cD(\sA)),\quad j(\sE,\sF)=j_2\circ j_1(\sE,\sF).
\]
Since $j_1$ is a quasi-homeomorphism that maps $(\sE',\sF')$ to $j_1(\sE',\sF')$,  it suffices to show that $j_2$ is a quasi-homoemorphism that maps $j_1(\sE',\sF')$ to $j(\sE',\sF')=j_2\circ j_1(\sE',\sF')$.  To accomplish this,  consider a $j_1\sA$-nest $\{\widehat{F}^1_n\subset \widehat{E}_1: n\geq 1\}$ such that $j_2$ is a topological homeomorphism from $\widehat{F}^1_n$ to $ j_2(\widehat{F}^1_n)$.  Let $\widehat{G}_1$ be the $j_1\sA$-quasi-open subset of $\widehat{E}_1$ as specified in Theorem~\ref{THM51} for $(\sE',\sF')$ and $(\sA,\cD(\sA))$.  According to Remark~\ref{RM49},  $\{\widehat{F}^1_n\cap \widehat{G}_1:n\geq 1\}$ forms a $j_1\sE'$-nest.  Consider another $j_1\sE'$-nest $\{\widehat{K}^1_n: n\geq 1\}$ consisting of compact sets in $\widehat{G}_1$,  and define $\widehat{K}'^{1}_n:=\widehat{K}^1_n\cap \widehat{F}^1_n\cap \widehat{G}_1$.  Then $\{\widehat{K}'^1_n:n\geq 1\}$ is also a $j_1\sE'$-nest,  and $j_2: \widehat{K}'^1_n\rightarrow \widehat{K}_n:=j_2(\widehat{K}'^1_n)$ is a topological homoemorphism.  It can be verified that $\{\widehat{K}_n:n\geq 1\}$ is a $j\sE'$-nest by following the argument in \eqref{eq:51-2}.  This completes the proof.
\end{proof}

From this point forward, we will further assume that both $(\sA,\cD(\sA))$ and $(\sE,\sF)$ are quasi-regular on $L^2(E,m)$.  Let $X$ denote a Borel right process properly associated with $(\sE,\sF)$.  According to Corollary~\ref{COR46},  there exists an MF $M\in \text{MF}(X)$ with $m(E\setminus E_M)=0$ such that 
\[
(\sE',\sF')=(\sE^M,\sF^M).
\]
Recall that $\nu_M|_{E_M\times E_M}$ is a bivariate smooth measure with respect to $\sE^{E_M}$.  Specifically,  $\bar{\nu}_M:=\frac{1}{2}\left(\rho_M+\lambda_M\right)|_{E_M}$ is smooth with respect to $\sE^{E_M}$, and $\nu_M|_{E_M\times E_M\setminus d}\leq \nu|_{E_M\times E_M\setminus d}$,  where $\nu$ is the canonical measure of $X$.

\begin{theorem}\label{THM62}
Let $(\sE,\sF)$ and $(\sE^M,\sF^M)$ be given as above.  
A quasi-regular Dirichlet form $(\sA,\cD(\sA))$ on $L^2(E,m)$ is sandwiched between $\sE$ and $\sE^M$,  i.e.,  
\[
	\sE^M\preceq \sA\preceq \sE,
\]
if and only if there exists an MF $M'\in \text{MF}(X)$ with the properties 
\[
	E_M\subset E_{M'},\quad \nu_{M'}|_{E_M\times E_M}\leq \nu_M|_{E_M\times E_M}
\]
such that $(\sA,\cD(\sA))=(\sE^{M'},\sF^{M'})$. 
\end{theorem}
\begin{proof}
To demonstrate the sufficiency,  we note that  Corollary~\ref{COR34} indicate that $\sA\preceq \sE$.  For convenience,  let $G:=E_{M'}$ and $\sigma:=\nu_{M'}$.  Consider $f\in \sF^M$.  We will show that $f\in \cD(\sA)$.  Since $E_M\subset G$,  $\sE$-q.e.,  and $\sigma|_{E_M\times E_M}\leq \nu_M|_{E_M\times E_M}$,  it suffices to prove that
\begin{equation}\label{eq:61}
	\int_{E_M} \tilde{f}(x)^2\left(\sigma(dx,  G\setminus E_M)+\sigma(G\setminus E_M,dx)\right)<\infty.
\end{equation}
 In fact,  it follows from $f\in \sF^M\subset \sF^{E_M}$ and \eqref{eq:39} that 
\[
	\int_{E_M}\tilde{f}(x)^2\left(\nu(dx, G\setminus E_M)+\nu(G\setminus E_M,dx)\right)\leq \int_{G\times G}\left(\tilde{f}(x)-\tilde{f}(y) \right)^2\nu(dxdy)<\infty. 
\]
Note that $\sigma|_{G\times G\setminus d}\leq \nu|_{G\times G\setminus d}$.  Thus,  we can easily derive \eqref{eq:61}.  By utilizing $E_M\subset G$, $\sE$-q.e., and $\sF^M\subset \cD(\sA)$,  it is straightforward to verify that $\mathrm{b}\sF^M$ forms an algebraic ideal in $\mathrm{b}\cD(\sA)$.  The condition  $\sigma|_{E_M\times E_M}\leq \nu_M|_{E_M\times E_M}$ also implies that $\sE^M(f,g)\geq \sA(f,g)$ for all non-negative $f,g\in \sF^M$.  

Conversely,  let $(\sA,\cD(\sA))$ be a quasi-regular Dirichlet form such that $\sE^M\preceq \sA\preceq \sE$.  By applying Theorem~\ref{THM41}  to $\sA$ and $\sE$,  we obtain an MF $M'\in \text{MF}(X)$ with $m(E\setminus E_{M'})=0$  such that $(\sA,\cD(\sA))=(\sE^{M'},\sF^{M'})$.  Let $G:=E_{M'}$ and $\sigma:=\nu_{M'}$.   It follows from $\sF^M=\sF^{E_M}\cap L^2(E_M,\bar{\nu}_M)\subset \cD(\sA)\subset \sF^G$ that  $\sF^{E_M}\subset \sF^G$.  (Note that $\sF^M$ is $\sE_1$-dense in $\sF^{E_M}$.) Thus,  $E_M\subset G$,  $\sE$-q.e.  It remains to show 
\begin{equation}\label{eq:62-2}
\sigma|_{E_M\times E_M}\leq \nu_M|_{E_M\times E_M}.
\end{equation}
  In fact,  from {Corollary~\ref{COR33}~(2)},  we know that $E_M$ is an $\sA$-quasi-open set.  Consider the part Dirichlet form of $(\sA,\cD(\sA))$ on $E_M$:
\[
\begin{aligned}
	&\cD(\sA^{E_M})=\sF^{E_M}\cap L^2(E_M,\bar{\sigma}|_{E_M}), \\
	&\sA^{E_M}(f,g)=\sE(f,g)+\sigma|_{E_M\times E_M}(\tilde{f}\otimes \tilde{g}),\quad f,g\in \cD(\sA^{E_M}).  
\end{aligned}\]
Applying Theorems~\ref{THM24} and \ref{THM41} to $\sE^M$ and $\sA$,  we find that $\sE^M$ is strongly subordinate to $\sA^{E_M}$.  Moreover,  it follows from Proposition~\ref{PRO39} and Theorem~\ref{THM52} that there exists a bivariate smooth measure $\sigma'$ on $E_M\times E_M$ with respect to $\sA^{E_M}$ such that $\nu_M|_{E_M\times E_M}=\sigma|_{E_M\times E_M}+\sigma'$.  Particularly,  \eqref{eq:62-2} is established.  This completes the proof.
\end{proof}
\begin{remark}\label{RM63}
According to Proposition~\ref{PRO39},  the sandwiched Dirichlet form $(\sA,\cD(\sA))$ can also be expressed as
\begin{equation}\label{eq:62}
\begin{aligned}
	&\cD(\sA)=\sF^G\cap L^2(G,\bar{\sigma}),\\
	&\sA(f,g)=\sE(f,g)+\sigma(\tilde f\otimes \tilde g),\quad f,g\in \cD(\sA),
\end{aligned}
\end{equation}
for some $\sE$-quasi-open set $G$ such that $E_M\subset G$,  $\sE$-q.e.,  and a bivariate smooth measure $\sigma$ on $G\times G$ with respect to $\sE^G$ such that $\sigma|_{E_M\times E_M}\leq \nu_M|_{E_M\times E_M}$.
\end{remark}

As noted in Remark~\ref{RM311},  killing by MFs reduces to perturbation for local Dirichlet forms.  Therefore, we can state the following.

\begin{corollary}\label{COR64}
Assume that $(\sE,\sF)$ is local.  Then, the quasi-regular Dirichlet form $(\sA,\cD(\sA))$ is sandwiched between $\sE$ and $\sE^M$ if and only if there exists an $\sE$-quasi-open set $G$ with $E_M\subset G$, $\sE$-q.e., and a smooth measure $\mu$ on $G$ with respect to $\sE^G$ with $\mu|_{E_M}\leq \bar{\nu}_M$ such that
\[
	\cD(\sA)=\sF^G\cap L^2(G,\mu),\quad \sA(f,g)=\sE(f,g)+\int_G\tilde{f}\tilde{g}d\mu,\;f,g\in \cD(\sA).
\]
\end{corollary}

We have decomposed domination into a composition of strong subordination and Silverstein extension in Theorem~\ref{THM24}.  The following corollary focuses on the case of strong subordination,  providing an explicit characterization within this setting. The complementary case of Silverstein extension will be explored in detail in the subsequent two subsections. 

\begin{corollary}
Assume that $(\sE^M,\sF^M)$ is strongly subordinated to $(\sE,\sF)$,  which implies that $E_M=E$.  Then, the quasi-regular Dirichlet form $(\sA,\cD(\sA))$ is sandwiched between $(\sE,\sF)$ and $(\sE^M,\sF^M)$ if and only if there exists a positive measure $\sigma$ on $E\times E$ with $\sigma\leq \nu_M$ such that
\[
	\cD(\sA)=\sF\cap L^2(E,\bar{\sigma}),\quad \sA(f,g)=\sE(f,g)+\sigma(\tilde{f}\otimes \tilde{g}),\; f,g\in \cD(\sA).
\]
\end{corollary}
\begin{proof}
It is sufficient to observe that a positive measure $\sigma$ satisfying $\sigma \leq \nu_M$ is a bivariate smooth measure, as is $\nu_M$.
\end{proof}

\subsection{Sandwiched by Silverstein extension}

When $(\sE,\sF)$ is a Silverstein extension of $(\sE',\sF')=(\sE^M,\sF^M)$,  we have 
\[
	\nu_M|_{E_M\times E_M}\equiv 0,\quad (\sE',\sF')=(\sE^{E_M},\sF^{E_M}).
\]
The result in Theorem~\ref{THM62}  indicates that the sandwiched Dirichlet form can be expressed as $(\sA,\cD(\sA))=(\sE^{M'},\sF^{M'})$ for some $M'\in \text{MF}(X)$  such that 
\[
	E_M\subset E_{M'},\; \sE\text{-q.e.}, \quad \nu_{M'}|_{E_M\times E_M}=0.
\]
Particularly,  $\nu_{M'}$ only charges $\left((E_{M'}\setminus E_M)\times E_M\right)\cup \left(E_M\times (E_{M'}\setminus E_M) \right)$ (off diagonal) and $\{(x,x)\in d: x\in E_{M'}\setminus E_M\}$ (on diagonal).  

In the example below, we enumerate all sandwiched Dirichlet forms derived from the Dirichlet form of absorbing Brownian motion and one of its Silverstein extensions. This illustrates that the bivariate smooth measure $\sigma$ in the expression \eqref{eq:62} does not necessarily reduce to a smooth measure as stated in Corollary~\ref{COR64}.

\begin{example}
Consider the Sobolev spaces
\[
	H^1([0,1]):=\{f\in L^2([0,1]): f\text{ is absolutely continuous, and }f'\in L^2([0,1])\}
\]	
and $H^1_0([0,1]):=\{f\in H^1([0,1]): f(0)=f(1)=0\}$.  Let
\[
\begin{aligned}
	\sF&:=H^1([0,1]),\\
	\sE(f,g)&:=\frac{1}{2}\mathbf{D}(f,g)+(f(0)-f(1))(g(0)-g(1)),\quad f,g\in \sF,
\end{aligned}\]
where $\mathbf{D}(f,g):=\int_0^1f'(x)g'(x)dx$.  Additionally, define 
\[
\begin{aligned}
	\sF'&:=H^1_0([0,1]),\\
	\sE'(f,g)&:=\frac{1}{2}\mathbf{D}(f,g),\quad f,g\in \sF'.
\end{aligned}
\]
It is evident that $(\sE',\sF')$ is a regular, symmetric Dirichlet form on $L^2((0,1))$ properly associated with the absorbing Brownian motion on $(0,1)$,  and $(\sE,\sF)$ is a regular,  symmetric Dirichlet form on $L^2([0,1])$.  The canonical measure corresponding to $(\sE,\sF)$ is 
\[
	\nu=\delta_{\{(0,1)\}}+\delta_{\{(1,0)\}}.
\] 
Since the $\sE_1$-norm is equivalent to the $\mathbf{D}_1$-norm on $H^1([0,1])$,  it follows that every {single-point} set contained in $[0,1]$ is not $\sE$-polar,  all $\sE$-quasi-continuous functions are continuous,  and all $\sE$-quasi-open sets are open.
Moreover,  $(\sE',\sF')$ is quasi-regular on $L^2([0,1])$,  and $(\sE,\sF)$ is a Silverstein extension of $(\sE',\sF')$.  


Let $(\sA,\cD(\sA))$ be a quasi-regular (not necessarily symmetric) Dirichlet form on $L^2([0,1])$ that is sandwiched between $(\sE',\sF')$ and $(\sE,\sF)$.  In its representation \eqref{eq:62},  the ($\sE$-quasi-)open set $G$ has four possibilities:
\[
	(0,1),\quad [0,1),\quad (0,1], \quad [0,1].
\]
The first case, where $G=(0,1)$,  is straightforward to address: We must have $\sigma\equiv 0$ in \eqref{eq:62}.  In other words,  the sandwiched Dirichlet form is identical to $(\sE',\sF')$.  The second and third cases are analogous,  so it suffices to discuss the second case,  where $G=[0,1)$.  The bivariate smooth measure $\sigma$ on $[0,1)\times [0,1)$ satisfies
\[
	\sigma|_{(0,1)\times (0,1)}\equiv 0,\quad \sigma|_{\left([0,1)\times [0,1)\right)\setminus d}\leq \nu|_{\left([0,1)\times [0,1)\right)\setminus d}.
\]  
From these conditions, {we conclude that $\sigma$ vanishes outside $\{(0,0)\}$.}  Particularly, 
\[
\begin{aligned}
	\cD(\sA)&=\{f\in H^1([0,1]): f(1)=0\}, \\
	 \sA(f,g)&=\frac{1}{2}\mathbf{D}(f,g)+c\cdot f(0)g(0),\quad f,g\in \cD(\sA),
\end{aligned}\]
for some constant $c:=1+\sigma(\{(0,0)\})\geq 1$.

It remains to consider the final case where $G=[0,1]$.  Similarly,  $\sigma$ {vanishes outside} $\{(0,0),(0,1), (1,0), (1,1)\}$,  with $\beta_0:=\sigma(\{(0,1)\})\leq 1$ and $\beta_1:=\sigma(\{(1,0)\})\leq 1$.  Specifically,  we have
\[
\begin{aligned}
	\cD(\sA)&=H^1([0,1]),\\
	\sA(f,g)&=\sE(f,g)+\alpha_0f(0)g(0)+\alpha_1f(1)g(1)+\beta_0f(0)g(1)+\beta_1 f(1)g(0),
\end{aligned}\]
for $f,g\in H^1([0,1])$,  where $\alpha_0:=\sigma(\{(0,0)\})\geq 0$ and $\alpha_1:=\sigma(\{(1,1)\})\geq 0$.  Note that $(\sA,\cD(\sA))$ is not necessarily symmetric.  It is symmetric if and only if $\beta_0=\beta_1$($=:\beta$),  in which case $(\sA,\cD(\sA))$ admits the Beurling-Deny decomposition
\[
\begin{aligned}
	\sA(f,g)&=\frac{1}{2}\mathbf{D}(f,g)+(1-\beta)\cdot (f(0)-f(1))(g(0)-g(1)) \\
	&\qquad +(\alpha_0+\beta)f(0)g(0)+(\alpha_1+\beta)f(1)g(1),\quad f,g\in H^1([0,1]),
\end{aligned}\]
where $0\leq \beta\leq 1$.  
\end{example}

\subsection{Sandwiched by reflected Dirichlet space}

We close this section by considering a typical case examined in the existing literatures such as \cite{ACD24, KL23, S20}.  Assume further that both $(\sE',\sF')$ and $(\sE,\sF)$ are symmetric,  and $(\sE,\sF)$ is the \emph{active reflected Dirichlet space} of $(\sE',\sF')$ in the sense of \cite[Definition~6.4.4]{CF12}.  In order to maintain consistency with the notation employed in \cite{CF12}, we will (only in this subsection) denote $(\sE', \sF')$ as $(\sE, \sF)$ and refer to $(\sE, \sF)$ as $(\sE^\mathrm{ref}, \sF^\mathrm{ref}_a)$.  It is important to note that $(\sE^\mathrm{ref}, \sF^\mathrm{ref}_a)$ is a Silverstein extension of $(\sE,\sF)$;  see \cite[Theorem~6.6.3]{CF12}.

{As in Theorem~\ref{THM62}, we assume that $(\sE^{\mathrm{ref}}, \sF^{\mathrm{ref}}_a)$ is quasi-regular on $L^2(E,m)$. This is not true in general. Nevertheless, owing to the mapping $j$ constructed in Lemma~\ref{LM61}, the image Dirichlet form of $(\sE^{\mathrm{ref}}, \sF^{\mathrm{ref}}_a)$ under $j$ is quasi-regular.
In typical situations, the mapping $j$ arises naturally by enlarging the state space $E$ with an additional boundary. On this extended space, $\sF^{\mathrm{ref}}_a$ remains the active reflected Dirichlet space of $(\sE,\sF)$, while the active reflected Dirichlet form $(\sE^{\mathrm{ref}}, \sF^{\mathrm{ref}}_a)$ becomes quasi-regular.
For example, let $\Omega\subset\mathbb{R}^n$ be a bounded domain with Lipschitz boundary, and consider $(\sE,\sF)$ as the regular Dirichlet form on $L^2(\Omega)$ associated with the absorbing Brownian motion on $\Omega$. The active reflected Dirichlet form $(\sE^{\mathrm{ref}}, \sF^{\mathrm{ref}}_a)$ is not quasi-regular on $L^2(\Omega)$. However, if $(\sE,\sF)$ is viewed as a Dirichlet form on $L^2(\overline{\Omega})$, where $\overline{\Omega}$ denotes the closure of $\Omega$, then it becomes quasi-regular, and $(\sE^{\mathrm{ref}}, \sF^{\mathrm{ref}}_a)$ turns into a regular Dirichlet form on $L^2(\overline{\Omega})$, corresponding to the reflecting Brownian motion on $\overline{\Omega}$. See \cite[Example~6.6.12]{CF12} for further details.
}

Let $M$ be the MF corresponding to the domination $\sE\preceq \sE^\mathrm{ref}$.  Then, $(\sE,\sF)$ is the part Dirichlet form of $(\sE^\mathrm{ref}, \sF^\mathrm{ref}_a)$ on $E_M$.  The set $E\setminus E_M$ is precisely the ``boundary" (of $E_M$),  as considered in \cite[Definition~4.2]{KL23}.  The following result demonstrates that the Dirichlet form sandwiched between $(\sE,\sF)$ and its active reflected Dirichlet space can be expressed as a perturbation of the part Dirichlet form of $(\sE^\mathrm{ref},\sF^\mathrm{ref}_a)$ on some $\sE^\mathrm{ref}$-quasi-open set $G$ ($\supset E_M$,  $\sE^\mathrm{ref}$-q.e.),  thereby recovering the main result, Theorem~4.6, of \cite{KL23}. 

\begin{corollary}
{Assume that both $(\sE,\sF)$ and the active reflected Dirichlet form $(\sE^\mathrm{ref},\sF^\mathrm{ref}_a)$ are quasi-regular on $L^2(E,m)$.}
The quasi-regular (not necessarily symmetric) Dirichlet form $(\sA,\cD(\sA))$ is sandwiched between $(\sE,\sF)$ and $(\sE^\mathrm{ref},\sF^\mathrm{ref}_a)$ if and only if there exists an $\sE^\mathrm{ref}$-quasi-open set $G$ with $E_M\subset G$,  $\sE^\mathrm{ref}$-q.e.,  and a smooth measure $\mu$ with respect to $\sE^{\mathrm{ref},G}$ with $\mu(G\setminus E_M)=0$ such that
\[
\begin{aligned}
	&\cD(\sA)=\sF^{\mathrm{ref},G}_a\cap L^2(G,\mu),  \\
	&\sA(f,g)=\sE^\mathrm{ref}(f,g)+\int_G \tilde{f}\tilde{g}d\mu,\quad f,g\in \cD(\sA),
\end{aligned}\]
where $(\sE^{\mathrm{ref},G}, \sF^{\mathrm{ref},G}_a)$ is the part Dirichlet form of $(\sE^\mathrm{ref},\sF^\mathrm{ref}_a)$ on $G$.  Particularly,  $(\sA,\cD(\sA))$ is symmetric. 
\end{corollary}
\begin{proof}
It suffices to prove the necessity.  For convenience,  we take every function to be its $\sE^\mathrm{ref}$-quasi-continuous $m$-version.   According to, e.g.,  \cite[Proposition~6.4.1]{CF12},  $(\sE^\mathrm{ref},\sF^\text{ref}_a)$ admits the Beurling-Deny decomposition as follows: For $f\in \sF^\mathrm{ref}_a$,  
\[
	\sE^\mathrm{ref}(f,f)=\frac{1}{2}\tilde{\mu}^c_{\langle f\rangle}(E)+\frac{1}{2}\int_{E\times E\setminus d}(f(x)-f(y))^2\tilde J(dxdy)+\int_{E}f(x)^2\tilde \kappa(dx),
\]
where $\tilde{\mu}^c_{\langle f\rangle}$ denotes the energy measure defined in \cite[(4.3.8)]{CF12}.  For $f\in \sF$,  we have
\[
\begin{aligned}
	\sE(f,f)&=\sE^\mathrm{ref}(f,f) \\
	&=\frac{1}{2}\tilde{\mu}^c_{\langle f\rangle}(E)+\frac{1}{2}\int_{E_M\times E_M\setminus d}(f(x)-f(y))^2 J(dxdy)+\int_{E_M}f(x)^2 \kappa(dx),
\end{aligned}\]
where 
\[
J:=\tilde{J}|_{E_M\times E_M\setminus d},\quad \kappa:=\left(\tilde\kappa+\tilde J(dx, E\setminus E_M)\right)\big|_{E_M}.
\]
Thus,  $J$ is the jumping measure of $(\sE,\sF)$ and $\kappa$ serves as its killing measure.  Utilizing \cite[(4.3.34)]{CF12},  we further conclude that 
\begin{equation}\label{eq:64}
	\mu^c_{\langle f\rangle}=\tilde{\mu}^c_{\langle f\rangle}\big|_{E_M},\quad f\in \sF,
\end{equation}
where $\mu_{\langle f\rangle}^c$ denotes the energy measure of $f\in \sF$ for the Dirichlet form $(\sE,\sF)$. 
 Let $\overset{\circ}{\sF}_\mathrm{loc}$ denote the local Dirichlet space of $\sF$,  as defined in \cite[(4.3.31)]{CF12}.  The measure $\mu^c_{\langle f\rangle}$ is well-defined for each $f\in \overset{\circ}{\sF}_\mathrm{loc}$ using \cite[Theorem~4.3.10]{CF12}.  According to the definition of the reflected Dirichlet space (see \cite[Definition~6.4.4]{CF12}),  we have
\[
	\sF^\mathrm{ref}_a=\{f\in \overset{\circ}{\sF}_\mathrm{loc}\cap L^2(E,m):\widehat{\sE}(f,f)<\infty\},
\]
where for $f\in \overset{\circ}{\sF}_\mathrm{loc}$,
\[
	\widehat{\sE}(f,f):=\frac{1}{2}\mu^c_{\langle f\rangle}(E_M)+\frac{1}{2}\int_{E_M\times E_M\setminus d}(f(x)-f(y))^2 J(dxdy)+\int_{E_M}f(x)^2 \kappa(dx),
\]
and $\sE^\mathrm{ref}(f,f)=\widehat{\sE}(f,f)$ for $f\in \sF^\mathrm{ref}_a$. 
We will prove that
\begin{equation}\label{eq:63}
\begin{aligned}	\widehat{\sE}^{(c)}(f,g)&:=\frac{1}{2}\mu^{c}_{\langle f,g\rangle}(E_M) \\
&=\frac{1}{8}\left(\mu^c_{\langle f+g\rangle}(E_M)-\mu^c_{\langle f-g\rangle}(E_M) \right),\quad f,g\in \sF^\mathrm{ref}_a\subset \overset{\circ}{\sF}_\mathrm{loc}
\end{aligned}\end{equation}
is strongly local in the sense of \cite[Proposition~6.4.1~(i)]{CF12}.  Once this is established,  the uniqueness of Beurling-Deny decomposition for $(\sE^\mathrm{ref},\sF^\text{ref}_a)$ implies that
\[
	\tilde J(E_M,E\setminus E_M)=\tilde J(E\setminus E_M,E_M)=\tilde J\left((E\setminus E_M)\times (E\setminus E_M)\setminus d\right)=0.
\]
Note that $\tilde J$ is the canonical measure of $(\sE^\mathrm{ref},\sF^\mathrm{ref}_a)$.  Consequently,  the bivariate smooth measure $\sigma$ in the expression \eqref{eq:62} for $(\sA,\cD(\sA))$ reduces to a smooth measure $\mu$ such that $\mu(G\setminus E_M)=0$, since $\sigma|_{E_M\times E_M}=0$.

It remains to demonstrate that \eqref{eq:63} is strongly local.  Let $F$ denote the $\sE^\mathrm{ref}$-quasi-support of $f$, and assume that $g$ is constant in an $\sE^\mathrm{ref}$-quasi-open set $B$ with $F\subset B$. Consider a sequence of increasing $\sE$-quasi-open sets $\{G_n: n\geq 1\}$ such that $\bigcup_{n\geq 1}G_n=E_M$,  $\sE$-q.e.,  along with two sequences of functions $\{f_n: n\geq 1\}$ and $\{g_n:n\geq 1\}$ in $\sF$ such that $f_n=f$ and $g_n=g$,  $\sE$-q.e.  on $G_n$ for each $n\geq 1$.  We aim to prove that
\begin{equation}\label{eq:65}
	\mu^{c}_{\langle f_n,g_n\rangle}(G_n)=0 ,\quad n\geq 1.
\end{equation}
Corollary~\ref{COR33}~(2) indicates that $G_n\setminus F$ is $\sE^\mathrm{ref}$-quasi-open.  
Given that $\mu^c_{\langle f_n\rangle}|_{G_n}=\mu^c_{\langle f\rangle}|_{G_n}$,  it follows from \eqref{eq:64} and \cite[Proposition~4.3.1~(ii)]{CF12} that
\[
	\mu^c_{\langle f_n\rangle}(G_n\setminus F)=\tilde{\mu}^c_{\langle f\rangle}(G_n\setminus F)=0.
\]
Thus,  we have 
\[
	\left|\mu^c_{\langle f_n,g_n\rangle}(G_n\setminus F)\right|^2\leq \mu^c_{\langle f_n\rangle}(G_n\setminus F)\mu^c_{\langle g_n\rangle}(G_n\setminus F)=0.
\]
Since $g$ is constant on $G_n\cap B$ ($\supset G_n\cap F$),  we can similarly obtain that $\mu^c_{\langle g_n\rangle}(G_n\cap B)=\tilde{\mu}^c_{\langle g\rangle}(G_n\cap B)=0$,  which further implies $\mu^c_{\langle f_n,g_n\rangle}(G_n\cap F)=0$.  Therefore,  \eqref{eq:65} is established.  This completes the proof.
\end{proof} 

\section{Application to the Laplacian with Robin boundary conditions}\label{SEC7}

Let $\Omega\subset \bR^n$ (with $n\geq 1$) be a non-empty open set with the boundary $\Gamma:=\overline{\Omega}\setminus \Omega$. Consider the Sobolev space 
\[
	H^1(\Omega):=\{u\in L^2(\Omega): D_j u\in L^2(\Omega),j=1,2,\cdots,n\}
\]
with the norm 
\[
	\|u\|^2_{H^1(\Omega)}:=\|u\|^2_{L^2(\Omega)}+\sum_{j=1}^n \|D_ju\|^2_{L^2(\Omega)},
\]
where $D_ju=\frac{\partial u}{\partial x_j}$ represents the distributional derivative. For $u,v\in H^{1}(\Omega)$, define
\[
	\bD(u,v):=\sum_{j=1}^{n}\int_{\Omega}D_{j}u(x)D_{j}v(x)dx.
\] 
Moreover, we let 
\[
	\tilde{H}^{1}(\Omega):=\overline{H^{1}(\Omega)\cap C(\overline{\Omega})}^{H^{1}(\Omega)},
\]
where $C(\overline{\Omega})$ denotes the space of all continuous real-valued functions on $\overline{\Omega}$, and 
\[
	H^{1}_{0}(\Omega):=\overline{C_c^{\infty}(\Omega)}^{H^{1}(\Omega)},
\]
where $C_{c}^{\infty}(\Omega)$ denotes the space of all infinitely differential functions on $\Omega$ with compact support. 

It is well known that $(\bD,\tilde{H}^{1}(\Omega))$ is a regular, symmetric Dirichlet form on $L^{2}(\overline{\Omega})$,  while $(\bD, H^{1}_{0}(\Omega))$ is a regular, symmetric Dirichlet form on $L^2(\Omega)$; see, e.g., \cite[\S2]{A03}. Denote by  $\overline{X}$ and $X^0$ the Hunt processes on $\overline\Omega$ and ${\Omega}$ corresponding to these forms, respectively.  For convenience, let $\text{Cap}$ be the $1$-capacity of $(\bD,\tilde{H}^{1}(\Omega))$,  and if the form symbol preceding the quasi-notion is omitted, it defaults to the quasi-notion corresponding to $(\bD,\tilde{H}^{1}(\Omega))$. For instance, $\tilde{u}\in \tilde{H}^1(\Omega)$ denotes the quasi-continuous version of $u$ with respect to $(\bD,\tilde{H}^{1}(\Omega))$, and the statement $\tilde{u}=0$ q.e. on $\Gamma$ asserts that $\tilde{u}$ is identically equal to $0$ on $\Gamma$ outside some polar set with respect to $(\bD,\tilde{H}^{1}(\Omega))$.

\subsection{Local Robin boundary}

In this subsection,  we aim to investigate the (not necessarily symmetric) Dirichlet form $(\sE,\sF)$ on $L^2(\Omega)$ ($=L^2(\overline{\Omega})$) that is sandwiched between $(\bD, H^{1}_{0}(\Omega))$ and $(\bD,\tilde{H}^{1}(\Omega))$,  i.e.,
\begin{equation}\label{eq:72}
(\bD, H^{1}_{0}(\Omega))\preceq (\sE,\sF)\preceq (\bD,\tilde{H}^{1}(\Omega)).
\end{equation}
 To avoid trivial cases, we always assume that $H^1_0(\Omega)\neq \tilde H^1(\Omega)$,  meaning that $\Gamma$ is not polar (see,  e.g., \cite[Proposition~2.5]{A03}). 
 
 \begin{definition}\label{DEF71}
A positive Borel measure $\mu$ on $\Gamma_\mu\subset \Gamma$ is called \emph{quasi-admissible} if it satisfies the following conditions:
 \begin{itemize}
 \item[(1)] $B_\mu:=\Gamma\setminus \Gamma_\mu$ is quasi-closed.
 \item[(2)] $\mu$,  regarded as a measure on $\overline{\Omega}\setminus B_\mu$ with $\mu(\Omega):=0$,  is smooth with respect to the part Dirichlet form of $(\bD,\tilde{H}^1(\Omega))$  on $G_\mu:=\overline{\Omega}\setminus B_\mu$. 
\end{itemize}
 \end{definition}
 \begin{remark}
It is important to emphasize that the defining set $\Gamma_\mu$ of $\mu$ plays a crucial role in determining admissibility. For two admissible measures, even if their trivial extensions on $\overline{\Omega}$ are identical, the Dirichlet forms defined by these measures in \eqref{eq:73} will differ if their defining sets are not q.e.  equal.
 \end{remark}
 
 For any quasi-admissible measure $\mu$ on $\Gamma_\mu$,  we define a quadratic form on $L^2(\overline{\Omega})$ as follows:
\begin{equation}\label{eq:73}
\begin{aligned}
	&\sF^{\mu}:=\left\{\tilde u\in \tilde H^1(\Omega): \tilde u=0\text{ q.e. on }B_\mu,\int_{\Gamma_\mu} \tilde u(x)^2\mu(dx)<\infty\right\}, \\
	&\sE^{\mu}(\tilde u,\tilde v):=\bD(\tilde u,\tilde v)+\int_{\Gamma_\mu} \tilde u(x)\tilde v(x)\mu(dx),\quad \tilde u,\tilde v\in \sF^{\mu}.
\end{aligned}
\end{equation}
These forms encapsulate all Dirichlet forms that are sandwiched between $(\bD, H^{1}_{0}(\Omega))$ and $(\bD,\tilde{H}^{1}(\Omega))$,  as demonstrated in the following result.  It should be emphasized that the quasi-regularity of $(\sE,\sF)$ is not essential to this characterization,  as we explained in Lemma~\ref{LM61}. 

\begin{theorem}\label{THM73}
A quasi-regular (not necessarily symmetric) Dirichlet form $(\sE,\sF)$ on $L^2(\overline \Omega)$ is sandwiched between $(\bD, H^{1}_{0}(\Omega))$ and $(\bD,\tilde{H}^{1}(\Omega))$ if and only if there exists a quasi-admissible measure $\mu$ such that $(\sE,\sF)=(\sE^\mu,\sF^\mu)$. 
\end{theorem}
\begin{proof}
For a quasi-admissible measure $\mu$,  $(\sE^\mu,\sF^\mu)$ is a quasi-regular Dirichlet form on $L^2(G_\mu)$ according to, e.g., \cite[Theorems~3.3.8 and 5.1.5]{CF12}.   It is also quasi-regular on $L^2(\overline{\Omega})$ for the same reason explained in the proof of Corollary~\ref{COR34}.  Conversely, the opposing statement is a consequence of Corollary~\ref{COR64}. 
\end{proof}

If $(\sE,\sF)$ is merely a coercive closed form that satisfies \eqref{eq:72},  then the Ouhabaz's domination criterion (the first statement of Lemma~\ref{LM21}) indicates that it must satisfy the Markovian property,  meaning that $(\sE,\sF)$ is automatically a Dirichlet form. Particularly,  we have the following.

\begin{corollary}\label{COR73}
Let $(\sE,\sF)$ be a coercive closed form on $L^2(\Omega)$ that satisfies \eqref{eq:72}.  Then $(\sE,\sF)$ is a local and symmetric Dirichlet form.
\end{corollary}
\begin{proof}
It is sufficient to note that the transformation $j$ in Lemma~\ref{LM61} keeps the local property of Dirichlet forms. 
\end{proof}
\begin{remark}
In \cite{AW03},  the authors presented Example 4.5 to demonstrate that the locality condition in \cite[Theorem~4.1]{AW03} is indispensable.  However,  upon closer inspection, the example fails to substantiate this claim because the associated $L^2$-semigroup $(T_t)_{t\geq 0}$ in this example is not positive; therefore, the Ouhabaz's domination criterion is not applicable in this context.  
 In fact,  the locality condition in \cite[Theorem~4.1]{AW03} can indeed be omitted,  as we will explain in Example~\ref{EXA75} that the admissible measures in \cite{AW03} are actually quasi-admissible.  Readers are also referred to \cite[\S3]{K18} for this correction. 
\end{remark}

Let us examine the \emph{admissible measure},  which has been investigated in various works,  including \cite{K18, A03, AW03}.  It will be demonstrated that all admissible measures are,  in fact, quasi-admissible.

\begin{example}\label{EXA75}
For a positive Borel measure $\mu$ on $\Gamma$,  let
 \[
 \Gamma_\mu:=\{z\in \Gamma: \exists r>0\text{ such that }\mu(\Gamma\cap B(z,r))<\infty\},
 \]
 where $B(z,r):=\{z\in \bR^n: |z|<r\}$.  In other words,  $\Gamma_\mu$ is the part of $\Gamma$ on which $\mu$ is locally finite.  The measure $\mu$ is called \emph{admissible} if $\text{Cap}(A)=0$ implies $\mu(A)=0$ for any Borel set $A\subset \Gamma_\mu$.  The main result,  Theorem~4.1,  of \cite{AW03} states that if $\mu$ is admissible, then
 \begin{equation}\label{eq:77}
 \begin{aligned}
	&\cD(a_\mu):=\left\{\tilde u\in \tilde H^1(\Omega): \tilde u=0\text{ q.e. on }\Gamma\setminus \Gamma_\mu,\int_{\Gamma_\mu} \tilde u(x)^2\mu(dx)<\infty\right\}, \\
	&a_{\mu}(\tilde u,\tilde v):=\bD(\tilde u,\tilde v)+\int_{\Gamma_\mu} \tilde u(x)\tilde v(x)\mu(dx),\quad \tilde u,\tilde v\in \cD(a_{\mu})
\end{aligned}
 \end{equation}
 is a Dirichlet form that satisfies \eqref{eq:72}. 
 
 In fact,  $Y:=\Omega\cup \Gamma_\mu$ is clearly an open subset of $\overline{\Omega}$,  and $\mu$ is a Radon measure on $Y$.  If $\mu$ is admissible,  then $\mu$ is a Radon smooth measure with respect to the part Dirichlet form of $(\bD,\tilde{H}^{1}(\Omega))$ on $Y$.  According to \cite[Theorems~3.3.9 and 5.1.6]{CF12},  $(a_\mu, \cD(a_\mu))$ is a regular Dirichlet form on $L^2(Y)$.  Particularly,  $B_\mu:=\Gamma\setminus \Gamma_\mu=\overline{\Omega}\setminus Y$ is closed and, consequently,  quasi-closed.  Therefore,  we conclude that $\mu$, when restricted to $\Gamma_\mu$,  is quasi-admissible. 
\end{example}
\begin{remark}
Consider the case $n\geq 2$, where each single-point subset of $\Gamma$ is a polar set.   If $\mu$ is a smooth measure with $\text{supp}[\mu]=\Gamma$,   it is always possible to construct another smooth measure $\mu'$,  which is equivalent to $\mu$,  such that $\mu'$ is nowhere Radon on $\Gamma$ in the sense that $\mu'(U)=\infty$ for any non-empty open subset $U$ of $\Gamma$ (see \cite[IV, Theorem~4.7]{MR92}).  Note that $\mu'$ with $\Gamma_{\mu'}=\Gamma$ is quasi-admissible,  and $(\sE^{\mu'},\sF^{\mu'})$, defined as \eqref{eq:73}, is a quasi-regular Dirichlet form on $L^2(\overline{\Omega})$, which differs from $(\bD,H^1_0(\Omega))$.  However,  the admissibility argument in this example leads to $(a_{\mu'},\cD(a_{\mu'}))=(\bD,H^1_0(\Omega))$ because the part of $\Gamma$, on which $\mu'$ is locally finite, is empty.
\end{remark}

We close this subsection by introducing a method for deriving a quasi-admissible measure from a positive Borel measure on $\Gamma$ that does not charge any polar sets. 

Let $\mu$ be a positive measure on $(\Gamma,\sB(\Gamma))$, which can be extended to $\overline{\Omega}$ in the way that $\mu(\Omega):=0$.  Assume that $\mu$ charges no polar sets, i.e., $\text{Cap}(A)=0$ implies $\mu(A)=0$ for any subset $A\subset \Gamma$.  Under these conditions,  the integral $\int_\Gamma \tilde{u}^2d\mu$ ($\leq \infty$) is well defined for any quasi-continuous function $\tilde u \in \tilde H^1(\Omega)$.  Consider a sequence $\{u_n\}\subset  \tilde H^1(\Omega)\cap L^2(\Gamma, \mu)$,  which is dense in $\tilde H^1(\Omega)\cap L^2(\Gamma, \mu)$ with respect to the norm $\|\cdot\|_{H^1(\Omega)}$.  We define the set
\[
	B_\mu:=\bigcap_{n=1}^\infty\{\tilde{u}_n=0\}. 
\]
Note that $B_\mu\subset \Gamma$,  q.e.,  and $B_\mu$ is independent of the choice of the sequence $\{u_n\}$,  as can be demonstrated by the following fact.

\begin{lemma}\label{LM71}
Let $\mu$ be a positive measure on $(\Gamma,\sB(\Gamma))$ such that $\text{Cap}(A)=0$ implies $\mu(A)=0$ for any subset $A\subset \Gamma$.  Then $B_\mu$ is a quasi-closed set such that $\tilde{u}=0$,  q.e. on $B_\mu$ for any $u\in \tilde H^1(\Omega)\cap L^2(\Gamma,\mu)$. 
Furthermore,  the restriction of $\mu$ to $\Gamma_\mu:=\Gamma\setminus B_\mu$ is quasi-admissible.
\end{lemma}
\begin{proof}
For any $u\in \tilde{H}^1(\Omega)\cap L^2(\Gamma,\mu)$, there exists a subsequence $\{u_{n_k}\}$ such that $\tilde u_{n_k}\rightarrow \tilde{u}$, q.e.  Thus,  $\tilde{u}=0$,  q.e. on $B_\mu$.  For the second part,  it suffices to show that $\mu$,  as a measure on $\overline{\Omega}\setminus B_\mu$, is smooth with respect to the part Dirichlet form of $(\bD,\tilde{H}^1(\Omega))$  on $\overline{\Omega}\setminus B_\mu$.  This fact has been established in \cite[Proposition~2.13]{RS95}. 
\end{proof}

\subsection{Non-local Robin boundary}

Now we turn to examine the Laplacian with non-local Robin boundary condition,  which has been discussed in \cite{OA21} and the references cited therein.   

\begin{definition}
Let $\kappa$ be a positive Borel measure on $\Gamma_{\kappa,\theta}(\subset \Gamma)$,  and let $\theta$ be a positive,  symmetric measure on $(\Gamma_{\kappa,\theta}\times \Gamma_{\kappa,\theta})\setminus d$,  where $d$ denotes the diagonal of $\Gamma \times \Gamma$. 
The pair $(\kappa, \theta)$ is called \emph{quasi-admissible} if the following conditions hold:
\begin{itemize}
\item[(1)] $B_{\kappa,\theta}:=\Gamma\setminus \Gamma_{\kappa,\theta}$ is quasi-closed.
\item[(2)] $\kappa+\bar\theta$ is smooth with respect to the part Dirichlet form of $(\bD, \tilde{H}^1(\Omega))$ on $\overline{\Omega}\setminus B_{\kappa,\theta}$,  where $\bar{\theta}(dx):=\theta(dx,\Gamma_{\kappa, \theta})$.  
\end{itemize}
\end{definition}
\begin{remark}
The measure $\kappa$ with $\Gamma_\kappa:=\Gamma_{\kappa,\theta}$ is quasi-admissible in the sense of Definition~\ref{DEF71}.  If $\theta=0$, then $(\sE^{\kappa,\theta},\sF^{\kappa,\theta})$ as presented in \eqref{eq:74} is identical to $(\sE^\kappa, \sF^\kappa)$ defined in \eqref{eq:73} with $\mu=\kappa$. 
\end{remark}

For any quasi-admissible pair $(\kappa, \theta)$,  we define the symmetric quadratic form on $L^2(\overline{\Omega})$ as follows:
\begin{equation}\label{eq:74}
\begin{aligned}
	&\sF^{\kappa,\theta}:=\bigg\{\tilde u\in \tilde H^1(\Omega): \tilde u=0\text{ q.e. on }B_{\kappa,\theta}, \\
	&\qquad \qquad\quad  \int_{\Gamma_{\kappa,\theta}\times \Gamma_{\kappa,\theta}\setminus d}(\tilde{u}(x)-\tilde u(y))^2\theta(dxdy)+\int_{\Gamma_{\kappa,\theta}} \tilde u(x)^2\kappa(dx)<\infty\bigg\}, \\
	&\sE^{\kappa,\theta}(\tilde u,\tilde v):=\bD(\tilde u,\tilde v)+\int_{\Gamma_{\kappa,\theta}\times \Gamma_{\kappa,\theta}\setminus d}(\tilde{u}(x)-\tilde u(y))(\tilde{v}(x)-\tilde v(y))\theta(dxdy) \\
	&\qquad \qquad\quad+\int_{\Gamma_{\kappa,\theta}} \tilde u(x)\tilde v(x)\kappa(dx),\quad \tilde u,\tilde v\in \sF^{\kappa,\theta}.
\end{aligned}
\end{equation}
The following fact is elementary.

\begin{lemma}
Let $(\kappa, \theta)$ be a quasi-admissible pair.  Then $(\sE^{\kappa,\theta},\sF^{\kappa,\theta})$ is a quasi-regular and symmetric Dirichlet form on $L^2(\overline{\Omega})$.  
\end{lemma}
\begin{proof}
Since $H^1_0(\Omega)\subset \sF^{\kappa,\theta}$,  it follows that $\sF^{\kappa,\theta}$ is dense in $L^2(\overline{\Omega})$.  Clearly,  $(\sE^{\kappa,\theta},\sF^{\kappa,\theta})$ is a symmetric,  positive form that satisfies the Markovian property.  To show that it is closed,  consider an $\sE^{\kappa,\theta}_1$-Cauchy sequence $\{\tilde{u}_n\}\subset \sF^{\kappa,\theta}$.  This sequence is also Cauchy in both $\tilde{H}^1(\Omega)$ and $L^2(\Gamma_{\kappa,\theta},\kappa)$,  and moreover,  $f_n(x,y):=\tilde{u}_n(x)-\tilde{u}_n(y)$ is Cauchy in $L^2(\Gamma_{\kappa,\theta}\times \Gamma_{\kappa,\theta}\setminus d,  \theta)$.  Hence, there exist $\tilde{u}\in \tilde{H}^1(\Omega)$,  $v\in L^2(\Gamma_{\kappa,\theta},\kappa)$, and $f\in L^2(\Gamma_{\kappa,\theta}\times \Gamma_{\kappa,\theta}\setminus d,  \theta)$ such that $\tilde{u}_n$ converges to $\tilde{u}$ and $v$ with respect to the $H^1(\Omega)$ and $L^2(\Gamma_{\kappa,\theta},\kappa)$-norms,  respectively,  and $f_n\rightarrow f$ with respect to the $L^2(\Gamma_{\kappa,\theta}\times \Gamma_{\kappa,\theta}\setminus d,  \theta)$-norm.  By taking a subsequence if necessary,  we can assume that $\tilde{u}_n\rightarrow \tilde{u}$, q.e.,  (which indicates that $\tilde{u}=$ q.e. on $B_{\kappa,\theta}$,)  $\tilde{u}_n\rightarrow v$,  $\kappa$-a.e., and $f_n\rightarrow f$,  $\theta$-a.e.  Since $\kappa+\bar{\theta}$ charges no polar sets,  we obtain that $\tilde{u}=v$, $\kappa$-a.e.,  and $f(x,y)=\tilde{u}(x)-\tilde{u}(y)$ for $\theta$-a.e. $(x,y)$.  Particularly,  $\tilde{u}\in \sF^{\kappa,\theta}$ and $\sE^{\kappa,\theta}_1(\tilde{u}_n-\tilde{u},\tilde{u}_n-\tilde{u})\rightarrow 0$.  This establishes the closedness of $(\sE^{\kappa,\theta},\sF^{\kappa,\theta})$.  

It remains to demonstrate that $(\sE^{\kappa,\theta},\sF^{\kappa,\theta})$ is quasi-regular on $L^2(\overline{\Omega})$.   Let $\mu:=\kappa+4\bar{\theta}$,  which is quasi-admissible in the sense of Definition~\ref{DEF71} with $\Gamma_\mu=\Gamma_{\kappa,\theta}$,  and define $(\sE^\mu,\sF^\mu)$ as in \eqref{eq:73}.  Notably,  we have
\[
	\sF^\mu\subset \sF^{\kappa,\theta},\quad \sE^{\kappa,\theta}(\tilde u,\tilde u)\leq \sE^\mu(\tilde u,\tilde u),\; \forall \tilde u\in \sF^\mu.  
\]
This implies that any $\sE^\mu$-nest is also an $\sE^{\kappa,\theta}$-nest.  Similarly,  consider the part Dirichlet form $(\sA,\sG)$ of $(\bD,\tilde{H}^1(\Omega))$ on $G:=\overline{\Omega}\setminus B_{\kappa,\theta}$.  From 
\[
	\sF^{\kappa,\theta}\subset \sG,\quad \sA(\tilde{u},\tilde{u})\leq \sE^{\kappa,\theta}(\tilde{u},\tilde{u}),\;\forall \tilde{u}\in \sF^{\kappa,\theta},  
\]
it follows that any $\sE^{\kappa,\theta}$-nest is also an $\sA$-nest.  According to \cite[Theorem~5.1.4]{CF12},  any $\sA$-nest is an $\sE^\mu$-nest.  Thus,  the quasi-notions for $\sE^{\kappa,\theta},\sE^\mu$ and $\sA$ are equivalent.  Therefore,  the quasi-regularity of $(\sE^{\kappa,\theta},\sF^{\kappa,\theta})$ follows directly from the quasi-regularity of $(\sA,\sG)$ and $(\sE^\mu,\sF^\mu)$.  This completes the proof.
\end{proof}

Analogously,  the admissible pair considered in \cite{OA21} is quasi-admissible,  as explained in the following example.

\begin{example}\label{EXA712}
Consider a positive Borel measure $\kappa$ on $\Gamma$ and a symmetric,  positive Borel measure $\theta$ on $\Gamma\times \Gamma \setminus d$.  Let $\Gamma_{\kappa,\theta}$ be the part of $\Gamma$ on which $\kappa$ and $\hat{\theta}$ are locally finite,  i.e.,
\begin{equation}\label{eq:76}
	\Gamma_{\kappa,\theta}:=\{z\in \Gamma: \exists r>0\text{ such that }\kappa(\Gamma\cap B(z,r))+\hat{\theta}(\Gamma\cap B(z,r))<\infty\},
\end{equation}
where $\hat\theta(dx):=\theta(dx,\Gamma)$.  The pair $(\kappa, \theta)$ is called \emph{admissible} if $\kappa(A)+\hat{\theta}(A)=0$ for any polar set $A\subset \Gamma_{\kappa,\theta}$.  For such a pair,  the key findings from \cite[Theorems~3.2 and 3.3]{OA21} indicate that
\[
\begin{aligned}
	&\cD(\sA^{\kappa,\theta}):=\bigg\{u\in H^1(\Omega)\cap C_c(\overline{\Omega}): u=0\text{ on }\Gamma\setminus \Gamma_{\kappa,\theta},\\&\qquad \qquad\quad\int_{\Gamma\times \Gamma\setminus d}({u}(x)- u(y))^2\theta(dxdy)+\int_{\Gamma}  u(x)^2\kappa(dx)<\infty\bigg\}, \\
	&\sA^{\kappa,\theta}( u, v):=\bD( u, v)+\int_{\Gamma\times \Gamma\setminus d}({u}(x)- u(y))({v}(x)- v(y))\theta(dxdy) \\
	&\qquad \qquad\quad+\int_{\Gamma}  u(x) v(x)\kappa(dx),\quad  u, v\in \cD(\sA^{\kappa,\theta})
\end{aligned}
\]
is closable,  and its closure is given by
\[
\begin{aligned}
	&\overline{\cD(\sA^{\kappa,\theta})}=
	\bigg\{\tilde u\in \tilde H^1(\Omega): \tilde u=0\text{ q.e.  on }\Gamma\setminus \Gamma_{\kappa,\theta},\\&\qquad \qquad\quad\int_{\Gamma\times \Gamma\setminus d}(\tilde{u}(x)- \tilde u(y))^2\theta(dxdy)+\int_{\Gamma}  \tilde u(x)^2\kappa(dx)<\infty\bigg\}.
\end{aligned}
\]

In fact,  $Y:=\Omega\cup \Gamma_{\kappa, \theta}$ is clearly an open subset of $\overline{\Omega}$,  and $\kappa+\hat{\theta}$ is a Radon measure on $Y$.  Thus,  $\kappa+\hat{\theta}$ is a Radon smooth measure on $Y$ with respect to the part Dirichlet form of $(\bD,\tilde{H}^1(\Omega))$ on $Y$.  It is noteworthy that $\cD(\sA^{\kappa,\theta})\subset C_0(Y)$,  where $C_0(Y)$ represents the family of all continuous functions on $Y$ that vanish at infinity (see \cite[page 132]{F99}),   and $\cD(\sA^{\kappa,\theta})$ is dense in $C_0(Y)$ with respect to the uniform norm by the Stone-Weierstrass theorem.  Particularly,  $(\sA^{\kappa,\theta},\overline{\cD(\sA^{\kappa,\theta})})$ is a regular,  symmetric Dirichlet form on $L^2(Y)$,  which admits the Beurling-Deny decomposition for $\tilde{u}, \tilde v\in \overline{\cD(\sA^{\kappa,\theta})}$ as follows:
\[
\begin{aligned}
\sA^{\kappa,\theta}(\tilde u,\tilde v)&=\bD(\tilde u,\tilde v)+\int_{\Gamma_{\kappa,\theta}\times \Gamma_{\kappa,\theta}\setminus d}(\tilde{u}(x)-\tilde u(y))(\tilde{v}(x)-\tilde v(y))\theta(dxdy) \\
	&\qquad \qquad\quad+\int_{\Gamma_{\kappa,\theta}} \tilde u(x)\tilde v(x)\left(\kappa+2\hat{\theta}_1\right)(dx),
\end{aligned}\]
where $\hat{\theta}_1(dx):={\theta}(dx,\Gamma\setminus \Gamma_{\kappa,\theta})$.  

It is straightforward to verify that $\kappa':=(\kappa+2\hat{\theta}_1)|_{\Gamma_{\kappa,\theta}}$ and $\theta'=\theta|_{\Gamma_{\kappa,\theta}\times \Gamma_{\kappa,\theta}\setminus d}$ comprise a quasi-admissible pair with $\Gamma_{\kappa',\theta'}=\Gamma_{\kappa,\theta}$,  and $(\sE^{\kappa',\theta'},\sF^{\kappa',\theta'})$,  as defined in \eqref{eq:74}, is identical to $(\sA^{\kappa,\theta},\overline{\cD(\sA^{\kappa,\theta})})$. 
\end{example}

Clearly, $(\bD,H^1_0(\Omega))$ is dominated by $(\sE^{\kappa,\theta},\sF^{\kappa,\theta})$,  while $(\sE^{\kappa,\theta},\sF^{\kappa,\theta})$ is not dominated by $(\bD, \tilde{H}^1(\Omega))$ due to Corollary~\ref{COR73}.  The following result characterizes all symmetric Dirichlet forms that are sandwiched between $(\bD,H^1_0(\Omega))$ and $(\sE^{\kappa,\theta},\sF^{\kappa,\theta})$.  

\begin{theorem}\label{THM713}
Let $(\kappa, \theta)$ be a quasi-admissible pair. 
A quasi-regular and symmetric Dirichlet form $(\sE,\sF)$ on $L^2(\overline{\Omega})$ is sandwiched between $(\bD,H^1_0(\Omega))$ and $(\sE^{\kappa,\theta},\sF^{\kappa,\theta})$  if and only if there exists another quasi-admissible pair $(\kappa',\theta')$ with the properties
\begin{equation}\label{eq:75}
\begin{aligned}
	&\Gamma_{\kappa',\theta'}\subset \Gamma_{\kappa,\theta},\quad \theta'\leq \theta|_{\Gamma_{\kappa',\theta'}\times \Gamma_{\kappa',\theta'}\setminus d},\\ &\kappa'\geq \kappa|_{\Gamma_{\kappa',\theta'}}+2(\theta-\theta')(dx,\Gamma_{\kappa',\theta'})|_{\Gamma_{\kappa',\theta'}}+2\theta(dx,\Gamma_{\kappa,\theta}\setminus \Gamma_{\kappa',\theta'})|_{\Gamma_{\kappa',\theta'}}
\end{aligned}\end{equation}
such that $(\sE,\sF)=(\sE^{\kappa',\theta'},\sF^{\kappa',\theta'})$.  
\end{theorem}
\begin{proof}
The sufficiency can be verified straightforwardly; we will now demonstrate the necessity.  According to Theorem~\ref{THM62} and Remark~\ref{RM63},  there exists an $\sE^{\kappa,\theta}$-quasi-open set $G$ with $\Omega\subset G\subset G_{\kappa,\theta}:=\Omega\cup \Gamma_{\kappa,\theta}$ along with a (symmetric) bivariate smooth measure $\sigma$ on $G\times G$ with respect to the part Dirichlet form $\left(\sE^1, \sF^1\right)$ of $\left(\sE^{\kappa, \theta}, \sF^{\kappa, \theta}\right)$ on $G$ such that $\sigma(\Omega\times \Omega)=0$ and
\[
\begin{aligned}
	&\sF=\sF^1\cap L^2(G,\bar{\sigma}),\\
	&\sE(u,v)=\sE^1(u,v)+\sigma(\tilde{u}\otimes \tilde v),\quad u,v\in \sF. 
\end{aligned}\]
Recall that the bivariate smooth measure $\sigma$ satisfies that $\bar{\sigma}(dx)=\sigma(dx,G)$ is smooth with respect to $\sE^1$ and $\sigma|_{G\times G\setminus d}\leq 2\theta|_{G\times G\setminus d}$.  Consequently, $\sigma(\Omega\times (G\setminus \Omega))=\sigma((G\setminus \Omega)\times \Omega)=0$,  which implies $\bar{\sigma}(\Omega)=0$. 

A straightforward computation shows that for $u,v\in \sF^1\cap L^2(G,\bar{\sigma})$,
\[
\begin{aligned}
	&\sE^1(u,v)+\sigma(\tilde{u}\otimes \tilde{v})=\bD(u,v)+\int_{G\setminus \Omega}\tilde{u}(x)\tilde{v}(x)\left(\kappa+2\theta(\cdot, \Gamma_{\kappa, \theta}\setminus G)+\bar{\sigma} \right)(dx) \\
	&\quad+\int_{(G\setminus \Omega)\times (G\setminus \Omega)\setminus d}(\tilde{u}(x)-\tilde{u}(y))(\tilde{v}(x)-\tilde{v}(y))\left(\theta-\frac{1}{2}\sigma\right)(dxdy).
\end{aligned}\]
Define 
\[
\theta':=\left(\theta-\frac{1}{2}\sigma\right)|_{(G\setminus \Omega)\times (G\setminus \Omega)\setminus d},\quad \kappa':=\left(\kappa+2\theta(\cdot, \Gamma_{\kappa, \theta}\setminus G)+\bar{\sigma}\right)|_{G\setminus \Omega}
\]
with 
\[
	\Gamma_{\kappa',\theta'}:=G\setminus \Omega. 
\]
It is evident that they satisfy \eqref{eq:75}. 

Let us verify that $(\kappa',\theta')$ is quasi-admissible.  Note that the quasi-notion of $\sE^{\kappa,\theta}$ corresponds to that of the part Dirichlet form of $(\bD, \tilde H^1(\Omega))$ on $G_{\kappa, \theta}$.  As stated in Corollary~\ref{COR33}~(2), $G$ is quasi-open,  and hence, $\Gamma\setminus \Gamma_{\kappa',\theta'}=\overline{\Omega}\setminus G$ is quasi-closed.  Since $\kappa+\bar{\theta}$ is smooth with respect to the part Dirichlet form of $(\bD,\tilde{H}^1(\Omega))$ on $G_{\kappa, \theta}$,  it follows from the remarks preceding Corollary~\ref{COR33} that the restriction of $\kappa+\bar{\theta}$ to $G$ is smooth with respect to the part Dirichlet form of $(\bD,\tilde{H}^1(\Omega))$ on $G$.  It is notable that $\bar{\sigma}$ is also smooth with respect to  the part Dirichlet form of $(\bD,\tilde{H}^1(\Omega))$ on $G$,  as it is smooth with respect to $\sE^1$.  Therefore,  we deduce that $\kappa'+\bar{\theta}'$ is smooth with respect to  the part Dirichlet form of $(\bD,\tilde{H}^1(\Omega))$ on $G$,  where $\bar{\theta}'(dx):=\theta'(dx,\Gamma_{\kappa',\theta'})$.  This verifies the quasi-admissibility of $(\kappa',\theta')$.  

It remains to demonstrate that $\sF=\sF^{\kappa',\theta'}$.  Note that
\[
\begin{aligned}
\sF^1&=\bigg\{\tilde u\in \tilde H^1(\Omega): \tilde u=0\text{ q.e. on }\Gamma\setminus \Gamma_{\kappa',\theta'}, \\
	&\qquad \qquad  \int_{\Gamma_{\kappa,\theta}\times \Gamma_{\kappa,\theta}\setminus d}(\tilde{u}(x)-\tilde u(y))^2\theta(dxdy)+\int_{\Gamma_{\kappa,\theta}} \tilde u(x)^2\kappa(dx)<\infty\bigg\} \\
&=\bigg\{\tilde u\in \tilde H^1(\Omega): \tilde u=0\text{ q.e. on }\Gamma\setminus \Gamma_{\kappa',\theta'},   \int_{\Gamma_{\kappa',\theta'}\times \Gamma_{\kappa',\theta'}\setminus d}(\tilde{u}(x)-\tilde u(y))^2\theta(dxdy)\\
	&\qquad \qquad+\int_{\Gamma_{\kappa',\theta'}} \tilde u(x)^2\left(\kappa(dx)+2\theta(dx,\Gamma_{\kappa,\theta}\setminus \Gamma_{\kappa',\theta'})\right)<\infty\bigg\}.
\end{aligned}\]
For $\tilde{u}\in \sF^1\cap L^2(G,\bar{\sigma})$,  we have
\[
\begin{aligned}
	&\int_{\Gamma_{\kappa',\theta'}\times \Gamma_{\kappa',\theta'}\setminus d}(\tilde{u}(x)-\tilde u(y))^2\theta'(dxdy)+\int_{\Gamma_{\kappa',\theta'}} \tilde u(x)^2\kappa'(dx) \\
	=&\int_{\Gamma_{\kappa',\theta'}}\tilde{u}(x)\tilde{u}(y)\sigma(dxdy)+\int_{\Gamma_{\kappa',\theta'}\times \Gamma_{\kappa',\theta'}\setminus d}(\tilde{u}(x)-\tilde u(y))^2\theta(dxdy)\\
	&\qquad \qquad+\int_{\Gamma_{\kappa',\theta'}} \tilde u(x)^2\left(\kappa(dx)+2\theta(dx,\Gamma_{\kappa,\theta}\setminus \Gamma_{\kappa',\theta'})\right)\\
	<& \infty,
\end{aligned}\]
since $\left| \int_{\Gamma_{\kappa',\theta'}}\tilde{u}(x)\tilde{u}(y)\sigma(dxdy)\right|\leq \int_{\Gamma_{\kappa',\sigma'}} \tilde{u}(x)^2\bar{\sigma}(dx)<\infty$.  This establishes that $\sF=\sF^1\cap L^2(G,\bar{\sigma})\subset \sF^{\kappa',\theta'}$.  The opposite inclusion can be verified similarly.  This completes the proof.
\end{proof}

One may be interested in local Dirichlet forms sandwiched between  $(\bD,H^1_0(\Omega))$ and $(\sE^{\kappa,\theta},\sF^{\kappa,\theta})$.  According to Theorem~\ref{THM713},  these forms are characterized as follows.

\begin{corollary}\label{COR714}
Let $(\kappa,\theta)$ be a quasi-admissible pair with $\bar{\theta}(\cdot):=\theta(\cdot, \Gamma_{\kappa,\theta})$.  
A quasi-regular and local Dirichlet form $(\sE,\sF)$ on $L^2(\overline{\Omega})$ is sandwiched between $(\bD,H^1_0(\Omega))$ and $(\sE^{\kappa,\theta},\sF^{\kappa,\theta})$ if and only if there exists a quasi-admissible measure $\mu$ with the properties
\[
\begin{aligned}
	&\Gamma_{\mu}\subset \Gamma_{\kappa,\theta},\quad \mu\geq (\kappa+2\bar\theta)|_{\Gamma_{\mu}}
\end{aligned}\]
such that $(\sE,\sF)=(\sE^\mu,\sF^\mu)$,  as defined in \eqref{eq:73}.
\end{corollary}
\begin{proof}
It suffices to note that the Dirichlet form $(\sE,\sF)$ in Theorem~\ref{THM713} is local if and only if $\theta'=0$.
\end{proof}

\subsection{Resolution of an open problem}\label{SEC73}

We close this section with a special consideration of the admissible pair $(\kappa,\theta)$ in Example~\ref{EXA712}.  Recall that $\Gamma_{\kappa,\theta}$ is defined as \eqref{eq:76},  $\hat{\theta}(\cdot):=\theta(\cdot, \Gamma)$,  and the regular Dirichlet form $(\sA^{\kappa,\theta},\overline{\cD(\sA^{\kappa,\theta})})$ corresponds to the quasi-admissible pair of measures
\[
	\kappa'=(\kappa+2{\theta}(\cdot    ,\Gamma\setminus \Gamma_{\kappa,\theta}))|_{\Gamma_{\kappa,\theta}},\quad \theta'=\theta|_{\Gamma_{\kappa,\theta}\times \Gamma_{\kappa,\theta}\setminus d}
\]  
with the defining set $\Gamma_{\kappa',\theta'}=\Gamma_{\kappa,\theta}$.  

Let $\mu_0:=\kappa+2\hat{\theta}$,   which is an admissible measure with $\Gamma_{\mu_0}=\Gamma_{\kappa,\theta}$ as discussed in Example~\ref{EXA75}.  Let $(a_{\mu_0},\cD(a_{\mu_0}))$ be defined as \eqref{eq:77} (with $\mu=\mu_0$).  Since $\Gamma_{\mu_0}=\Gamma_{\kappa',\theta'}$ and $\mu_0|_{\Gamma_{\mu_0}}=\kappa'+2\bar{\theta}'$,  where $\bar\theta'(\cdot):=\theta'(\cdot, \Gamma_{\kappa',\theta'})=\theta(\cdot, \Gamma_{\kappa,\theta})$,  it follows from Corollary~\ref{COR714} that $(a_{\mu_0},\cD(a_{\mu_0}))$ is sandwiched between $(\bD,H^1_0(\Omega))$ and $(\sA^{\kappa,\theta},\overline{\cD(\sA^{\kappa,\theta})})$.  The following result establishes that it is the largest among the local Dirichlet forms sandwiched between $(\bD, H^1_0(\Omega))$ and $(\sA^{\kappa,\theta}, \overline{\cD(\sA^{\kappa,\theta})})$. 

\begin{corollary}
If $(\sE,\sF)$ is a quasi-regular and local Dirichlet form on $L^2(\overline{\Omega})$ that is sandwiched between $(\bD,H^1_0(\Omega))$ and $(\sA^{\kappa,\theta}, \overline{\cD(\sA^{\kappa,\theta})})$,  then $(\sE,\sF)$ is dominated by $(a_{\mu_0},\cD(a_{\mu_0}))$.   
\end{corollary}
\begin{proof}
According to Corollary~\ref{COR714},   it holds that $(\sE,\sF)=(\sE^\mu,\sF^\mu)$ for some quasi-admissible $\mu$ such that $\Gamma_{\mu}\subset \Gamma_{\mu_0}$ and $\mu\geq \mu_0|_{\Gamma_\mu}$.  Consequently, the conclusion follows from Corollary~\ref{COR64}.
\end{proof}

Particularly,  if $\mu$ is an admissible measure such that $(a_{\mu},\cD(a_{\mu}))$ is sandwiched between $(\bD,H^1_0(\Omega))$ and $(\sA^{\kappa,\theta}, \overline{\cD(\sA^{\kappa,\theta})})$,  then $(a_{\mu},\cD(a_{\mu}))$ is dominated by $(a_{\mu_0},\cD(a_{\mu_0}))$.  This conclusion {solves} the open problem posed at the end of \cite{OA21}.

\appendix

\section{Quasi-regularity of Dirichlet forms}\label{SEC3}

In this appendix,  we recall the fundamental concepts related to the quasi-regularity of Dirichlet forms.  Let $E$ be a Hausdorff topological space with the Borel $\sigma$-algebra $\mathscr{B}(E)$ assumed to be generated by the continuous functions on $E$,  and let $m$ be a $\sigma$-finite measure on $E$ with support $\mathrm{supp}[m]=E$.  Denote by $\sB^*(E)$ be the universal completion of $\sB(E)$.  

\subsection{Quasi-regularity}

Let $(\sE,\sF)$ be a Dirichlet form on $L^2(E,m)$.  An increasing sequence of closed subsets $\{F_n:n\geq 1\}$ of $E$ is called an \emph{$\sE$-nest} if $\bigcup_{n\geq 1}\sF_{F_n}$ is $\sE_1$-dense in $\sF$,  where $\sF_{F_n}:=\{f\in \sF: f=0,m\text{-a.e. on }E\setminus F_n\}$.  A set $N\subset E$ is called \emph{$\sE$-polar} if there exists an $\sE$-nest $\{F_n:n\geq 1\}$ such that $N\subset E\setminus \left(\bigcup_{n\geq 1}F_n \right)$.  A statement is said to hold in the sense of $\sE$-quasi-everywhere (abbreviated as \emph{$\sE$-q.e.}) if it holds outside an $\sE$-polar set.   A function $f$ on $E$ is called \emph{$\sE$-quasi-continuous} if there exists an $\sE$-nest $\{F_n:n\geq 1\}$ such that $f|_{F_n}$ is finite and continuous on $F_n$ for each $n\geq 1$,  denoted by $f\in C(\{F_n\})$.  If $f=g$, $m$-a.e.,  and $g$ is $\sE$-quasi-continuous,  then $g$ is referred to as an $\sE$-quasi-continuous $m$-version of $f$.  A set $G\subset E$ is termed \emph{$\sE$-quasi-open} if there exists an $\sE$-nest $\{F_n:n\geq 1\}$ such that $G\cap F_n$ is an open subset of $F_n$ with respect to the relative topology for each $n\geq 1$.  The complement of an $\sE$-quasi-open set is called \emph{$\sE$-quasi-closed}. 
 It is important to emphasize that these concepts are related solely to the symmetric part of $\sE$.  

\begin{definition}
A Dirichlet form $(\sE,\sF)$ on $L^2(E,m)$ is called \emph{quasi-regular} if the following conditions hold:
\begin{itemize}
\item[(1)] There exists an $\sE$-nest $\{K_n:n\geq 1\}$ consisting of compact sets.
\item[(2)] There exists an $\sE_1$-dense subset of $\sF$ whose elements have $\sE$-quasi-continuous $m$-versions.
\item[(3)] There exists $\{f_k:k\geq 1\}\subset \sF$ having $\sE$-quasi-continuous $m$-versions $\{\tilde{f}_k:k\geq 1\}$ and an $\sE$-polar set $N$ such that $\{\tilde{f}_k:k\geq 1\}$ separates the points of $E\setminus N$.
\end{itemize}
It is called \emph{regular} if $E$ is a locally compact separable metric space,  $m$ is a fully supported Radon measure on $E$,  and $\sF\cap C_c(E)$ is $\sE_1$-dense in $\sF$ as well as uniformly dense in $C_c(E)$.  It is noteworthy that a regular Dirichlet form is always quasi-regular.
\end{definition}

It is well known that there exists an \emph{$m$-tight}  right process
\[
	X=\left(\Omega, \mathcal{F}, \mathcal{F}_t, X_t, \theta_t,\bP_x\right)
\]
 on $E$ \emph{properly} associated with a quasi-regular Dirichlet form $(\sE,\sF)$ in the sense that $P_tf$ is an $\sE$-quasi-continuous $m$-version of $T_tf$ for any $t\geq 0$ and $f\in L^2(E,m)\cap \mathrm{p}\sB(E)$,  where $P_tf(x):=\bP_xf(X_t)$ is the transition semigroup of $X$ and $(T_t)_{t\geq 0}$ is the $L^2$-semigroup corresponding to $(\sE,\sF)$  (see \cite[IV,  Theorem~3.5]{MR92}).  Denote by $(U_\alpha)_{\alpha>0}$ the resolvent of $X$,  defined as
\[
 U_\alpha f(x):=\bP_x \int_0^\infty e^{-\alpha t}f(X_t)dt
\]
for all $x\in E$ and $f\in \mathrm{b}\sB(E)$. The notation related to general Markov processes can be found in,  e.g., \cite{BG68} and \cite{S88}.  The symbol $\Delta$ represents the cemetery,  and $\zeta:=\inf\{t>0:X_t=\Delta\}$ denotes the lifetime of $X$.  Note that the $m$-tightness of $X$ means that there exists a sequence of increasing compact sets $\{K_n:n\geq 1\}$ such that
\begin{equation}\label{eq:32}
	\bP_m\left(\lim_{n\rightarrow \infty}T_{E\setminus K_n}<\zeta\right)=0,
\end{equation}
where $\bP_m:=\int_E m(dx)\bP_x$ and $T_{A}:=\inf\{t>0:X_t\in A\}$ denotes the first hitting time of any nearly Borel measurable set $A$ with respect to $X$.  

Without loss of generality,  we may further assume that $E$ is a Lusin space (see \cite[IV,  Remark~3.2~(iii)]{MR92}),  and that $X$ is Borel measurable in the sense that $P_tf\in \sB(E)$ for all $t\geq 0$ and $f\in \mathrm{b}\sB(E)$ (see, e.g., \cite[Corollary~3.23]{F01}).  In this context,  $X$ is referred to as a \emph{Borel right process}.  As established in \cite{W72},   $X_{t-}$ exists in $E$ for $t<\zeta$,  $\bP_m$-a.s. (This fact will be utilized only in \S~\ref{SEC43}.) 

A set $N\subset E$ is called \emph{$m$-polar} (for $X$) if there exists a nearly Borel measurable set $\tilde{N}\supset N$ such that $\bP_m(T_{\tilde{N}}<\infty)=0$.  A statement holds in the sense of \emph{q.e.}  if it holds outside an $m$-polar set.  A nearly Borel measurable set $A\subset E$ is said to be ($X$-)\emph{invariant} if for every $x\in A$,  $\bP_x\left(T_{E\setminus A}<\infty\right)=0$.  The restriction of a Borel right process $X$ to an invariant set is a (not necessarily Borel) right process (see \cite[(12.30)]{S88}),  whereas its restriction to a Borel invariant set remains a Borel right process (see,  e.g., \cite[Lemma~A.1.27]{CF12}).  
A subset $N\subset E$ is called an \emph{$m$-inessential set} (for $X$) if $m(N)=0$ and $E\setminus N$ is $X$-invariant.  Obviously,  an $m$-inessential set is $m$-polar.   Note that any $m$-polar set is contained in an $m$-inessential Borel set (see,  e.g., \cite[Theorem~A.2.15]{CF12}).  A numerical function $f$ defined q.e.  on $E$ is called \emph{finely continuous q.e.} if there exists an $m$-inessential  set $N$ such that $f|_{E\setminus N}$ is nearly Borel measurable and finely continuous with respect to the restricted right process $X|_{E\setminus N}$.  A set $G\subset E$ is called \emph{q.e.  finely open} if there exists an $m$-inessential set $N$ such that $G\setminus N$ is a nearly Borel measurable and finely open set for $X|_{E\setminus N}$.  

The following lemma summarizes the relationships between several $\sE$-quasi-notions and concepts related to $X$.

\begin{lemma}\label{LM32}
Let $X$ be a Borel right process properly associated with the quasi-regular Dirichlet form $(\sE,\sF)$ on $L^2(E,m)$.  Then the following statements hold:
\begin{itemize} 
\item[(1)] A set $N\subset E$ is $\sE$-polar,  if and only if it is $m$-polar. 
\item[(2)] An increasing sequence of closed subsets $\{F_n:n\geq 1\}$ of $E$ is an $\sE$-nest if and only if 
\begin{equation}\label{eq:32-2}
		\bP_x\left(\lim_{n\rightarrow \infty}T_{E\setminus F_n}<\zeta\right)=0
\end{equation}
for q.e.  $x\in E$ (equivalently,  for $m$-a.e.  $x\in E$).
\item[(3)] If $f$ is $\sE$-quasi-continuous,  then $f$ is finely continuous q.e.  Conversely,  if $f\in \sF$ is finely continuous q.e., then $f$ is $\sE$-quasi-continuous. 
\item[(4)] A set $G\subset E$ is $\sE$-quasi-open if and only if it is q.e. finely open.
\end{itemize}
\end{lemma}
\begin{proof}
 For the first statement,  see \cite[IV, Theorem~5.29(i)]{MR92}.  The second statement is a consequence of \cite[IV, Theorem~5.4 and Proposition~5.30]{MR92}.  The third statement can be established by repeating the proof of \cite[Theorem~3.1.7]{CF12}. (This proof relies solely on the essential properties of Borel right processes and the conclusions from the previous two assertions. Additionally,  it is necessary to adjust $N_0$ in this proof to an $m$-polar set $N'_0$ such that \eqref{eq:32-2} holds for $x\notin N'_0$.)
 
 The necessity of the fourth statement can be concluded by repeating the proof of \cite[Theorem~3.3.3]{CF12}.  In particular,  if $G$ is $\sE$-quasi-open,  there exists an $m$-inessential Borel set $N$ such that $G\setminus N\in \sB(E\setminus N)$ is finely open with respect to $X|_{E\setminus N}$. Conversely,  let $G$ be q.e. finely open, and let $N$ be an $m$-inessential set for $X$ such that $G_1:=G\setminus N$ is nearly Borel measurable and finely open with respect to $X|_{E\setminus N}$.  Take $f\in \sB(E)\cap L^2(E,m)$ that is strictly positive on $E$,  and define
 \[
 	g(x):=\bP_x \int_0^{T_{E\setminus G_1}}e^{-t}f(X_t)dt=U_1f(x)-\bP_x \left(e^{-T_{E\setminus G_1}}U_1f(X_{T_{E\setminus G_1}}) \right).
 \]
Note that $\bH^1_{E\setminus G_1}U_1f(x):=\bP_x \left(e^{-T_{E\setminus G_1}}U_1f(X_{T_{E\setminus G_1}}) \right)$ is $1$-excessive with respect to $X$ (see \cite[Lemma~A.2.4~(ii)]{CF12}),  and $\bH^1_{E\setminus G_1}U_1f\leq U_1f\in \sF$.  It follows from,  e.g.,  \cite[Theorem~2.6]{MOR95} that $\bH^1_{E\setminus G_1}U_1f\in \sF$.  Thus,  the third statement indicates that both $\bH^1_{E\setminus G_1}U_1f$ and $g$ are $\sE$-quasi-continuous.  It is evident that $\{x\in E: g>0\}=G_1\cup (B_1\setminus B_1^r)$,  where $B_1:=E\setminus G_1$ and $B_1^r$ is the set of all regular points for $B_1$.  Note that $B_1\setminus B_1^r$ is $m$-polar,  and hence $\sE$-polar.  Since $\{g>0\}$ is $\sE$-quasi-open,  we can conclude that both $G_1$ and $G$ are $\sE$-quasi-open.  This completes the proof. 
\end{proof}
\begin{remark}
The third statement indicates the following fact: if $f\in \sF$ is finely continuous q.e.,  then there exists an $m$-inessential Borel set $N$ such that $f|_{E\setminus N}$ is not only nearly Borel measurable with respect to $X|_{E\setminus N}$ but also Borel measurable on $E\setminus N$. 
\end{remark}

\subsection{Quasi-homeomorphism}

Let $(\widehat{E},\sB(\widehat{E}))$ be a second topological space with the Borel measurable $\sigma$-algebra $\sB(\widehat{E})$,  and let $j:(E,\sB(E))\rightarrow (\widehat{E}, \sB(\widehat{E}))$ be a measurable map.  Define $\widehat{m}:=m\circ j^{-1}$,  the image measure of $m$ under $j$.  Then 
\[
	j^*:L^2(\widehat{E},\widehat{m})\rightarrow L^2(E,m),\quad \widehat f\mapsto  j^*\widehat f:=\widehat{f}\circ j
\]
is an isometry.  If $j^*$ is onto, i.e.,  $j^*L^2(\widehat{E},\widehat{m})=L^2(E,m)$,  then
\[
\begin{aligned}
	&\widehat{\sF}:=\left\{\widehat{f}\in L^2(\widehat{E},\widehat{m}): j^*\widehat{f}\in \sF\right\}, \\
	&\widehat{\sE}(\widehat{f},\widehat{g}):=\sE(j^*\widehat{f},j^*\widehat{g}),\quad \widehat{f},\widehat{g}\in \widehat{\sF}
\end{aligned}\]
is a Dirichlet form on $L^2(\widehat{E},\widehat{m})$,  referred to as the \emph{image Dirichlet form} of $(\sE,\sF)$ under $j$.  We denote $(\widehat{\sE},\widehat{\sF})$ as $j(\sE,\sF)$.  

\begin{definition}\label{DEF34}
Let $(\widehat{\sE},\widehat{\sF})$ be another Dirichlet form on $L^2(\widehat{E},\widehat{m})$ with $\text{supp}[\widehat{m}]=\widehat E$.  The Dirichlet form $(\sE,\sF)$ is called \emph{quasi-homeomorphic} to $(\widehat{\sE},\widehat{\sF})$ if there exists an $\sE$-nest $\{F_n:n\geq 1\}$, an $\widehat{\sE}$-nest $\{\widehat{F}_n: n\geq 1\}$ and a map $j:\bigcup_{n\geq 1}F_n\rightarrow \bigcup_{n\geq 1}\widehat{F}_n$ such that 
\begin{itemize}
\item[(1)] $j$ is a topological homeomorphism from $F_n$ to $\widehat{F}_n$ for each $n\geq 1$.
\item[(2)] $\widehat{m}=m\circ j^{-1}$.
\item[(3)] $(\widehat{\sE},\widehat{\sF})=j(\sE,\sF)$,  the image Dirichlet form of $(\sE,\sF)$ under $j$. 
\end{itemize}
Such a map $j$ is called a \emph{quasi-homeomorphism} from $(\sE,\sF)$ to $(\widehat{\sE},\widehat{\sF})$.  
\end{definition}

Note that a quasi-homeomorphism keeps the quasi-notions invariant; see,  e.g.,  \cite[Corollary~3.6]{CMR94}. 

{
\section*{Acknowledgement}

The authors would like to thank Professor Zhen-Qing Chen from the University of Washington. He brought to our attention that Silverstein has examined the subordination of semigroups in the symmetric case in \cite{S74}. This result is referred to as the Third Structure Theorem in \cite[Theorem 21.2]{S74}. {The authors also  wish to express sincere appreciation to the anonymous referees for their careful reading and many constructive comments, which have led to substantial improvements in both the exposition and the mathematical precision of the paper.}
}

\bibliographystyle{abbrv}
\bibliography{SubDom}

\end{document}